\documentclass[11pt,letterpaper]{article}
\usepackage[T1]{fontenc}
\usepackage{graphicx}
\usepackage{amsfonts,amssymb,amsbsy,amsmath,amsthm}
\usepackage{cite}
\usepackage{mathpazo}
\usepackage{tgpagella}

\usepackage[margin=0.85in]{geometry}
\usepackage{float}
\usepackage{mathtools}
\usepackage{booktabs}
\usepackage{siunitx}
\usepackage{color}
\graphicspath{{figures/}}
\usepackage{subcaption}
\usepackage{algorithm}
\usepackage{algpseudocode}
\usepackage{hyperref}
\hypersetup{
  colorlinks=true,
  linkcolor=blue,
  citecolor=blue,
  urlcolor=blue
}

\theoremstyle{plain}
\newtheorem{theorem}{Theorem}[section]

\theoremstyle{definition}
\newtheorem{definition}{Definition}[section]

\newtheorem{remark}{Remark}[section]
\newtheorem{proposition}{Proposition}[section]
\makeatletter
\renewenvironment{abstract}{%
  \normalfont\small
  \list{}{\labelwidth\z@
    \leftmargin4pc
    \itemindent\z@
    \parsep\z@ \@plus\p@
  }%
  \item[\hskip\labelsep\bfseries\abstractname.]%
}{%
  \endlist
}

\renewcommand\@maketitle{%
  \global\topskip8\p@\relax
  \newpage
  \null
  \vskip 1em%
  \begin{center}%
  \let\footnote\thanks
    {\LARGE\bfseries \@title \par}%
    \vskip 1.2em%
    {\large
      \lineskip .5em%
      \begin{tabular}[t]{c}%
        \@author
      \end{tabular}\par}%
    \ifx\@date\@empty\else\vskip 1em{\large \@date}\fi
  \end{center}%
  \par
  \vskip 1em}

\long\def\@makefntext#1{\parindent 1em
  \noindent{$\m@th^{\@thefnmark}$}#1}

\renewcommand\subsubsection{\@startsection{subsubsection}{3}{\z@}%
                                     {-2ex\@plus -1ex \@minus -.2ex}%
                                     {1ex \@plus .2ex}%
                                     {\normalfont\normalsize\bfseries}}
\makeatother

\newcommand{\ph}{\phi}%
\newcommand{\phn}{\phi^n}
\newcommand{\phnp}{\phi^{n+1}}%
\newcommand{\phNN}{\phi_\theta}
\newcommand{\grad}{\nabla}%
\newcommand{\norm}[1]{\left\| #1 \right\|}
\newcommand{\E}{\mathcal{F}}
\newcommand{\X}{\mathcal{X}}
\newcommand{\argmin}{\mathop{\mathrm{argmin}}}
\newcommand{\dd}{\mathrm{d}}
\newcommand{\ip}[2]{\left\langle #1,#2\right\rangle}

\newif\ifshowtodos\showtodostrue

\begin{document}
\raggedbottom

\title{High-Order Variable-Scaled Energetic Variational Neural Networks for Phase-Field Gradient Flows}

\author{%
  Xiaobo Jing\textsuperscript{1},\quad
  Leiyi Dong\textsuperscript{1},\quad
  Jia Zhao\textsuperscript{2,}\thanks{Corresponding author.
    \textit{Email addresses:} \texttt{xiaobo@seu.edu.cn} (X.~Jing);
    \texttt{220251948@seu.edu.cn} (L.~Dong);
    \texttt{jia.zhao@ua.edu} (J.~Zhao)}%
  \\[6pt]
  \normalsize\itshape
  \textsuperscript{1}School of Mathematics, Southeast University, Nanjing 210096, P.R. China.\\
  \normalsize\itshape
  \textsuperscript{2}Department of Mathematics, University of Alabama, Tuscaloosa, AL 35487, USA
}
\date{}

\maketitle
\thispagestyle{empty}

\noindent\hrulefill
\begin{abstract}
In this paper, we develop a high-order variable-scaled energetic variational neural network (VS-EVNN) method for phase-field gradient flows. Each discretization stage uses an updated-state network whose spatial inputs are centered and divided by the interface scale, while the sampled energy and movement penalties remain unchanged relative to the baseline EVNN formulation. This rescaling is motivated by a coordinate change in the energy that normalizes the gradient coefficient of the transformed functional. We combine the scaled representation with high-order, multi-stage variational extrapolation schemes and a heuristic adaptive time-stepping controller. Each stage is solved as a network minimization problem, in which the earlier stages serve as fixed anchors on the quadrature grid, yielding a neural realization of these multi-stage variational schemes, including a mobility-weighted, mass-conserving version for Cahn--Hilliard flows. The method is tested on one- and two-dimensional Allen--Cahn and Cahn--Hilliard flows, the latter in the mass-conserving metric. Under the reported optimization settings, the scaled networks yield lower solution errors than the unscaled EVNN baseline, and the higher-order schemes improve the accuracy in the test examples.
\end{abstract}

\vspace{-6pt}

\noindent\hrulefill

\section{Introduction} Gradient-flow partial differential equations (PDEs) appear throughout physics, materials science, and biology, where they describe systems that evolve to minimize an energy functional. The Allen--Cahn equation~\cite{AllenCahn1979}, the Cahn--Hilliard equation~\cite{CahnHilliard1958}, and mean curvature flow are standard examples that model phase separation, interface dynamics, and geometric evolution. Traditional numerical methods for these equations must contend with stiffness, the fine spatial resolution required near interfaces, and the cost of long-time simulations~\cite{Du2020, Shen2010}.

A solver with a mesh-free trial representation could be reused across geometries and problem instances without generating a new mesh for each configuration. Physics-informed neural networks (PINNs) provide one such route. The PINN framework, formulated in its modern form by Raissi et al.~\cite{Raissi2019PINN}, encodes PDE residuals directly into the loss function. Related neural PDE methods include variational formulations, such as the Deep Ritz method~\cite{E2018Deep} and domain-decomposition approaches that enforce flux continuity across subdomain interfaces~\cite{Jagtap2020cPINN}. Fourier feature embeddings~\cite{Tancik2020Fourier} and hash encodings~\cite{Muller2022InstantNGP} have also been used to represent high-frequency content.

Applying PINNs to phase-field models is difficult, especially when directly training a space--time network on the Allen--Cahn and Cahn--Hilliard equations. Thin moving interfaces, stiff nonlinearities, and long horizons create an uneven loss landscape, and a global-in-time fit may learn later times before earlier ones. Wight and Zhao~\cite{WightZhao2021} combined adaptive collocation near interfaces with time-segment marching. Mattey and Ghosh~\cite{Mattey2022bcPINN} proposed backward-compatible sequential training for the same equations, whereas causal residual weighting~\cite{Wang2024Causality} explicitly orders training over time.

Input transformations are a separate response to sharp spatial scales. Variable linear transformations have been used for thin-layer flow PINNs~\cite{Wu2024VLTPINN}, and VS-PINN~\cite{Ko2025VSPINN} rescales space--time inputs, analyzes the resulting training kernel, and includes an Allen--Cahn benchmark. More recent phase-field variants include PF-PINNs for coupled Allen--Cahn/Cahn--Hilliard systems~\cite{Chen2025PFPINNs}, data-driven discovery of phase-field dynamics~\cite{Katbar2025DataDriven}, and a discrete-time moving-contact-line PINN that combines variable scaling with dynamic boundary conditions~\cite{Chen2026MCLPINN}. Residual formulations can also incorporate energy information through adaptive window selection~\cite{Guo2024AdaptiveEnergyPINN} or an energy-dissipation penalty~\cite{Kutuk2025EnergyPINN}. These energy-informed variants incorporate energy information into training but do not, by themselves, establish an unconditional per-step energy inequality for inexactly optimized networks. The Energetic Variational Neural Network (EVNN) framework~\cite{Hu2024EVNN} defines each time step via energy minimization; transferring its variational energy inequality to computed networks likewise requires control over the inner solves.

The minimizing-movements viewpoint has several neural precedents beyond EVNN. The deep minimizing movement scheme of Park et al.~\cite{Park2023DMM} trains a network at each implicit step. Georgoulis, Loulakis, and Tsiourvas~\cite{Georgoulis2023Discrete} approximate discrete gradient flows with sequences of residual-type networks. Phase-Field DeepONet incorporates minimizing movements into an operator time stepper for Allen--Cahn and Cahn--Hilliard dynamics~\cite{Li2023PhaseFieldDeepONet}, while Mattey and Ghosh~\cite{Mattey2025SNN} use separable networks and prove a discrete energy result. Zheng, Wang, and Yang~\cite{Zheng2026Increment} formulate the neural step over solution increments and analyze optimization and error propagation on a Riemannian manifold. Their increment geometry is distinct from the fixed spatial-coordinate reparametrization considered here.

Other structure-preserving neural approaches make different choices. Energy-dissipative evolutionary DeepONets use a scalar auxiliary variable and preserve a modified discrete energy~\cite{Zhang2024EnergyDeepONet}. Feature-based DNN Galerkin methods use neural features to define trial spaces with semidiscrete energy dissipation, which are then paired with energy-stable time integrators~\cite{Tang2026DNNG}. Exact periodic neural layers are also available~\cite{DongNi2021Periodic}. For time discretization, the present work uses the multi-stage variational extrapolation of Zaitzeff, Esedo\u{g}lu, and Garikipati~\cite{Zaitzeff2020}, later extended to semi-implicit schemes~\cite{Zaitzeff2021} and general metrics~\cite{Han2023Metric}.

Energy stability has a long history in classical numerical methods for phase-field models. Because the energy dissipation law is the defining structural property of a gradient flow, many time discretizations have been designed to inherit it at the discrete level. These include convex splitting schemes that treat the convex and concave parts of the energy implicitly and explicitly~\cite{eyre1998unconditionally}, stabilized semi-implicit schemes that add a linear stabilization term to permit large time steps~\cite{Shen2010}, and exponential time differencing (ETD) schemes that integrate the stiff linear part exactly~\cite{Du2019ETD}. Invariant energy quadratization~\cite{Yang2016IEQ} and scalar auxiliary variable~\cite{Shen2018SAV} reformulations yield linear schemes with unconditional dissipation of a modified discrete energy. A related Lagrange multiplier approach~\cite{Cheng2020Lagrange} enforces dissipation of the original energy and requires an additional scalar nonlinear solve. These methods are efficient, and many are supported by rigorous error analysis, but they typically operate on a fixed spatial mesh or basis. The neural-network discretization considered here keeps the same minimizing-movements structure while replacing the spatial discretization with a mesh-free network ansatz, with energies and metric terms still evaluated on fixed quadrature grids.

This work addresses a scale mismatch between the network coordinates and the solution features in the EVNN inner optimization. A phase-field profile varies across an interface layer whose width is set by the small interface parameter, so an accurate trial function must produce order-one changes over a region far narrower than the domain. A coordinate network with generically initialized first-layer weights represents such thin transitions inefficiently during training. To remove the mismatch, we propose the Variable-Scaled EVNN (VS-EVNN), which evaluates each updated-state network using spatial coordinates centered on the domain and scaled by a factor proportional to the interface parameter. At the same time, the sampled energy, the quadrature, and the movement penalties of the minimizing-movements objective are unchanged. A change of coordinates in the energy shows that this choice normalizes the gradient coefficient of the transformed functional. At a fixed time step and spatial resolution, the scaled representation gives more accurate solutions than the baseline EVNN in our tests. We then combine the scheme with the high-order multi-stage variational time integration of~\cite{Zaitzeff2020} and with adaptive time-stepping within the same minimizing-movements structure. On the high-order side, the paper contributes the network realization of the multi-stage construction, in which each stage is solved by a separate warm-started network and the earlier stages are incorporated into the loss as fixed grid values. This realization extends to the mobility-weighted discrete $H^{-1}$ metric for Cahn--Hilliard flows. The paper also identifies which parts of the exact-solve energy-stability argument survive inexact network minimization, and it notes a sign correction for one coefficient in the six-stage third-order matrix in the preprint version of~\cite{Zaitzeff2020}. A set of one- and two-dimensional Allen--Cahn and Cahn--Hilliard experiments characterizes the relative contributions of temporal discretization error and the non-temporal error floor. The same scaled representation operates in the $H^{-1}$ metric with the mass constraint enforced by projection.

\section{Gradient-flow background and the EVNN discretization}

\subsection{Gradient flows and the choice of metric} Let \(\X\) be a Hilbert space with inner product \(\ip{\cdot}{\cdot}\) and norm \(\norm{\cdot}\), and let
\[
\E : \X \to (-\infty,+\infty]
\]
be an energy functional. The formal gradient flow of \(\E\) is
\begin{equation}
\partial_t \phi(t) = -\grad \E\bigl(\phi(t)\bigr),
\qquad t>0.
\label{eq:abstract-gf}
\end{equation}
When \(\E\) is Fr\'echet differentiable, \(\grad \E(\phi)\in \X\) is defined by
\[
D\E(\phi)[w] = \ip{\grad \E(\phi)}{w}
\qquad \text{for all } w\in \X.
\]
The choice of \(\X\) (equivalently, of the inner product used to identify the derivative \(D\E\) with a gradient) is part of the model. The same energy functional generates different dynamics in different metrics~\cite{Ambrosio2008}.

At the PDE level, the same structure is often written in variational derivative form. For a spatial field \(\phi(x,t)\) on a sufficiently smooth domain \(\Omega\), if
\[
\E(\phi)=\int_\Omega f(x,\phi,\grad \phi)\,\dd x,
\]
then taking \(\X=L^2(\Omega)\), with \(\ip{u}{v}=\int_\Omega uv\,\dd x\), identifies \(\grad\E(\phi)\) with the variational derivative \(\frac{\delta\E}{\delta\phi}\), and the \(L^2\)-gradient flow reads
\begin{equation}
\partial_t \phi = -\frac{\delta \E}{\delta \phi}.
\end{equation}
A standard example is the Allen--Cahn equation:
\[
\E(\phi)=\int_\Omega \left(\frac12 |\grad \phi|^2 + F(\phi)\right)\dd x,
\qquad
\frac{\delta \E}{\delta \phi} = -\Delta \phi + F'(\phi),
\]
so the gradient flow is
\[
\partial_t \phi = \Delta \phi - F'(\phi).
\]
Thus, the Allen--Cahn equation is the gradient flow of the diffuse-interface energy in \(\X=L^2(\Omega)\).

With a mobility operator \(K\), the flow takes the more general form
\begin{equation}
\partial_t \phi = -K\,\frac{\delta \E}{\delta \phi}.
\label{eq:mobility-gf}
\end{equation}
When \(K\) is symmetric and positive definite, \eqref{eq:mobility-gf} is again a gradient flow of the same energy, now in the Hilbert space \(\X\) whose inner product is weighted by \(K^{-1}\), namely \(\ip{u}{v}_{\X}=\ip{K^{-1}u}{v}_{L^2}\). When \(K\) is only positive semidefinite, the inner product is defined on the range of \(K\) through the pseudoinverse, \(\ip{u}{v}=\ip{K^{\dagger}u}{v}_{L^2}\). For the Laplacian mobility below, with periodic or homogeneous Neumann boundary conditions, this range is the mean-zero subspace. For example, in Cahn--Hilliard,
\[
\partial_t \phi = \Delta \mu,
\qquad
\mu = \frac{\delta \E}{\delta \phi},
\]
the mobility is \(K=-\Delta\) (with periodic or homogeneous Neumann boundary conditions), and the equation is the gradient flow of the same diffuse-interface energy in \(\X=H^{-1}(\Omega)\), with the homogeneous \(H^{-1}\) inner product defined on mean-zero functions by
\[
\ip{u}{v}_{H^{-1}} = \ip{(-\Delta)^{-1}u}{v}_{L^2},
\qquad
\norm{u}_{H^{-1}}^2 = \int_\Omega \bigl|\grad(-\Delta)^{-1}u\bigr|^2\,\dd x.
\]
Because Cahn--Hilliard conserves the spatial mean of \(\phi\), the dynamics evolve in an affine subspace of fixed mean. Differences such as \(\phi-\phi^n\) are therefore mean-zero, and their \(H^{-1}\) norms are well defined.

The choice of metric carries over to the time discretization. In the minimizing-movements scheme, the metric term \(\frac{1}{2\tau}\norm{\phi-\phi^n}^2\) must be measured in the \(\X\)-norm of the flow. This is the \(L^2\) norm for Allen--Cahn and the \(H^{-1}\) norm for Cahn--Hilliard-type flows~\eqref{eq:mobility-gf}~\cite{Ambrosio2008}. Throughout the paper, unsubscripted \(\norm{\cdot}\) and \(\ip{\cdot}{\cdot}\) denote the norm and inner product of the underlying space \(\X\). We write an explicit subscript, as in \(\norm{\cdot}_{L^2}\), when a concrete space is meant.

\subsection{Minimizing-movements scheme}
\label{subsec:minimizing-movements}

The minimizing-movements scheme is the variational implicit Euler discretization of \eqref{eq:abstract-gf}, also known as the proximal-map iteration. Given an initial state \(\phi^0\in \X\) and a time step \(\tau>0\), define \(\phi^{n+1}\) recursively by
\begin{equation} \label{eq:minimizing_movement}
 \phi^{n+1} \in \argmin_{\phi\in \X}
 \left\{
 \Phi_\tau^n(\phi)
 :=
 \frac{1}{2\tau}\norm{\phi-\phi^n}^2 + \E(\phi)
 \right\}.
\end{equation}
The first term penalizes motion away from the previous state, whereas the second drives the evolution downhill in energy. The metric term is a discrete dissipation rate, and the variational derivative of the energy supplies the generalized force, so the objective encodes the Onsager force--flux--mobility structure of the flow and realizes the Onsager variational principle at the time-discrete level~\cite{wang2020generalized}. For the abstract statement, the existence of a minimizer follows from the direct method when the objective is proper, coercive, and sequentially weakly lower semicontinuous on a weakly closed admissible subset of \(\X\). For minimization over the full finite-dimensional state space, continuity and coercivity of the sampled objective suffice.

Taking \(\phi=\phn\) as a competitor in \eqref{eq:minimizing_movement} immediately yields the discrete energy law.

\begin{proposition}[Energy Dissipation]
\label{prop:energy_dissipation} The minimizing-movements scheme satisfies
\begin{equation}
\E(\phnp) + \frac{1}{2\tau}\norm{\phnp - \phn}^2 \leq \E(\phn).
\end{equation}
In particular, $\E(\phnp) \leq \E(\phn)$ for all $n \geq 0$.
\end{proposition}

\begin{proof}
By definition, $\phnp$ minimizes the functional $\Phi_\tau^n$ of \eqref{eq:minimizing_movement}. Since $\Phi_\tau^n(\ph)=\E(\ph)+\frac{1}{2\tau}\norm{\ph-\phn}^2$, evaluating at $\ph=\phn$ as a competitor gives
\[
\E(\phnp) + \frac{1}{2\tau}\norm{\phnp - \phn}^2 = \Phi_\tau^n(\phnp) \leq \Phi_\tau^n(\phn) = \E(\phn),
\]
which is the stated inequality. Discarding the nonnegative metric term yields $\E(\phnp)\leq\E(\phn)$.
\end{proof}

The scheme is first-order accurate in time for an abstract finite-dimensional gradient system, which, in our setting, serves as an intermediate, exact spatial discretization. This formulation also avoids imposing an artificial global Lipschitz assumption on the unbounded Laplacian in the continuum Allen--Cahn equation.

\begin{theorem}[First-order temporal convergence for the discrete gradient system]
\label{thm:first_order}
Let $\X_h\simeq\mathbb{R}^N$ be equipped with a discrete inner product, and let $\E_h\in C^1(\X_h)$ be the corresponding discrete energy. Let $\ph_h\in C^2([0,T];\X_h)$ solve
\[
\partial_t\phi_h(t)=-\grad\E_h(\phi_h(t)).
\]
Assume that $\grad\E_h$ is Lipschitz continuous on a set $B_h\subset\X_h$ that contains the exact trajectory $\{\phi_h(t):0\le t\le T\}$ and the discrete iterates $\{\phi_h^n\}$,
\[
\norm{\grad\E_h(u)-\grad\E_h(w)} \le L_h\,\norm{u-w}
\qquad \text{for all } u,w\in B_h.
\]
Let $\{\phi_h^n\}$ be generated from $\phi_h^0=\phi_h(0)$ by the discrete minimizing-movements step
\[
\phi_h^{n+1}\in
\argmin_{\phi\in\X_h}
\left\{
\frac{1}{2\tau}\norm{\phi-\phi_h^n}^2+\E_h(\phi)
\right\},
\]
assuming that a minimizer exists at every step. Then, for every $\tau>0$ satisfying $\tau L_h\le 1/2$,
\begin{equation}
\max_{0\le n\le \lfloor T/\tau\rfloor}\,
\norm{\phi_h(t_n)-\phi_h^n}
\le
\frac{M_{2,h}\,T\,e^{2L_hT}}{2}\,\tau,
\qquad
M_{2,h}:=\sup_{t\in[0,T]}\norm{\partial_{tt}\phi_h(t)} .
\label{eq:first-order-bound}
\end{equation}
\end{theorem}

\begin{proof}

The Euler--Lagrange equation of the discrete step,
\begin{equation}
\frac{\phi_h^{n+1}-\phi_h^n}{\tau}
+\grad\E_h(\phi_h^{n+1})=0,
\label{eq:euler-lagrange}
\end{equation}
identifies the scheme as implicit Euler.

Since $\phi_h^{n+1}$ minimizes the differentiable discrete functional over $\X_h$, it is a critical point, and the Euler--Lagrange equation \eqref{eq:euler-lagrange} holds exactly.

A Taylor expansion with integral remainder about $t_{n+1}$ gives
\[
\phi_h(t_n)=\phi_h(t_{n+1})-\tau\,\partial_t\phi_h(t_{n+1})+R_h^n,
\qquad
\norm{R_h^n}\le\frac{\tau^2}{2}\,M_{2,h},
\]
and substituting $\partial_t\phi_h(t_{n+1})=-\grad\E_h(\phi_h(t_{n+1}))$ yields
\[
\phi_h(t_{n+1})
=
\phi_h(t_n)-\tau\,\grad\E_h(\phi_h(t_{n+1}))-R_h^n .
\]

Set $\xi^n:=\phi_h(t_n)-\phi_h^n$. Subtracting \eqref{eq:euler-lagrange} from the identity above gives
\[
\xi^{n+1}
=
\xi^n
-\tau\Bigl(\grad\E_h(\phi_h(t_{n+1}))
-\grad\E_h(\phi_h^{n+1})\Bigr)
-R_h^n .
\]
The Lipschitz bound on $B_h$ then gives
\[
\norm{\xi^{n+1}}
\le
\norm{\xi^n}+\tau L_h\norm{\xi^{n+1}}
+\frac{\tau^2}{2}M_{2,h}.
\]
For $\tau L_h\le 1/2$, we have $(1-\tau L_h)^{-1}\le 1+2\tau L_h\le e^{2\tau L_h}$, hence
\[
\norm{\xi^{n+1}}
\le
e^{2\tau L_h}
\Bigl(\norm{\xi^n}+\frac{\tau^2}{2}M_{2,h}\Bigr).
\]
Since $\xi^0=0$, iterating gives, for $n\tau\le T$,
\[
\norm{\xi^n}
\le
\frac{\tau^2}{2}M_{2,h}\sum_{k=1}^{n} e^{2k\tau L_h}
\le
\frac{\tau^2}{2}M_{2,h}\,n\,e^{2L_hT}
\le
\frac{M_{2,h}\,T\,e^{2L_hT}}{2}\,\tau,
\]
which is \eqref{eq:first-order-bound}.
\end{proof}

\subsection{Mass conservation and the discrete \texorpdfstring{$H^{-1}$}{H-1} metric for Cahn--Hilliard flows}
\label{subsec:ch-metric}

For the Cahn--Hilliard equation with constant mobility $M>0$, the abstract scheme \eqref{eq:minimizing_movement} covers conserved dynamics once the state space and the metric are chosen accordingly:
\begin{equation}
\partial_t \phi = M\Delta\mu,
\qquad
\mu = \frac{\delta \E}{\delta \phi} = -\varepsilon^2\Delta\phi + F'(\phi),
\label{eq:ch-pde}
\end{equation}
posed on a periodic domain $\Omega$ with the Ginzburg--Landau energy
\begin{equation}
\E(\phi)=\int_\Omega \Bigl(\frac{\varepsilon^2}{2}|\grad\phi|^2 + F(\phi)\Bigr)\dd x,
\qquad
F(\phi)=\frac14(\phi^2-1)^2 .
\label{eq:ch-energy}
\end{equation}
Integrating \eqref{eq:ch-pde} over $\Omega$ shows that the total mass is conserved,
\begin{equation}
\frac{\dd}{\dd t}\int_\Omega \phi\,\dd x = \int_\Omega M\Delta\mu\,\dd x = 0,
\end{equation}
so the natural state space is the affine mass manifold
\begin{equation}
\mathcal{M}_{m_0} = \Bigl\{\phi : \int_\Omega \phi\,\dd x = m_0|\Omega|\Bigr\},
\qquad
m_0 = \frac{1}{|\Omega|}\int_\Omega \phi_0\,\dd x .
\end{equation}
The minimizing-movements step for Cahn--Hilliard is the constrained problem
\begin{equation}
\phi^{n+1}\in\argmin_{\phi\in\mathcal{M}_{m_0}}
\Bigl\{\frac{1}{2M\tau}\norm{\phi-\phi^n}_{H^{-1}}^2 + \E(\phi)\Bigr\},
\label{eq:ch-minmove}
\end{equation}
where the constant mobility is absorbed into the metric as the factor $1/M$. Proposition~\ref{prop:energy_dissipation} applies verbatim, since its proof uses only the Hilbert-space structure and the competitor $\phi=\phi^n\in\mathcal{M}_{m_0}$. The scheme therefore dissipates the energy unconditionally and, by construction, conserves the mass exactly at every step.

The metric term in \eqref{eq:ch-minmove} requires the $H^{-1}$ norm of the mean-zero increment $\phi-\phn$. This norm is characterized through a Poisson problem. For a mean-zero function $f$, define $\psi_f$ by
\begin{equation}
-\Delta\psi_f = f \ \text{in } \Omega,
\qquad
\int_\Omega \psi_f\,\dd x = 0,
\end{equation}
with periodic boundary conditions. Then $\psi_f=(-\Delta)^{-1}f$ and
\begin{equation}
\norm{f}_{H^{-1}}^2
= \int_\Omega f\,\psi_f\,\dd x
= \int_\Omega |\grad\psi_f|^2\,\dd x .
\end{equation}

\subsection{Neural-network approximation and the baseline EVNN}
\label{sec:nn_approximation}

We approximate the solution $\ph$ by a coordinate network. At a given time level, the state is represented as a function of the spatial coordinates alone,
\begin{equation}
\phNN(x) = \mathrm{NN}_\theta(x), \qquad x\in\Omega\subset\mathbb{R}^d,
\end{equation}
where $\theta$ denotes the network parameters. Retraining the network at every time step is inherent to this sequential representation, and the schemes below exploit this structure.

The baseline EVNN method~\cite{Hu2024EVNN} uses a coordinate-based residual network (ResNet) with tanh activations to approximate $\ph$. At each time step, we solve the minimizing-movements problem
\begin{equation}
\theta^{n+1} \in \argmin_\theta \left\{
\E(\phNN) + \frac{1}{2\tau}\norm{\phNN - \phn}^2
\right\}
\end{equation}
using limited-memory Broyden--Fletcher--Goldfarb--Shanno (L-BFGS) optimization. Here $\norm{\cdot}$ denotes the norm on the underlying space $\X$, namely the $L^2$ norm for Allen--Cahn and the mobility-weighted $H^{-1}$ norm of Section~\ref{subsec:ch-metric} for Cahn--Hilliard. The initial condition is fitted via
\begin{equation}
\theta^0 \in \argmin_\theta \norm{\phNN - \ph_0}_{L^2}^2.
\end{equation}

For the Cahn--Hilliard step, the mass constraint in \eqref{eq:ch-minmove} is enforced explicitly via a mean-value correction to the network output. The raw output $\mathrm{NN}_\theta$ is replaced by $\phNN=\mathrm{NN}_\theta - \overline{\mathrm{NN}_\theta} + m_0$, where the overline denotes the mean over the quadrature grid. This projects the sampled ansatz onto the discrete counterpart of $\mathcal{M}_{m_0}$ exactly, and the continuum mass is conserved to the accuracy of the quadrature rule. The baseline EVNN loss for the Cahn--Hilliard step is then the sampled version of
\begin{equation}
\mathcal{L}(\theta)
= \frac{1}{2M\tau}\int_\Omega (\phNN-\phn)\,\psi_\theta\,\dd x + \E(\phNN),
\qquad
-\Delta\psi_\theta = \phNN-\phn,
\quad
\int_\Omega \psi_\theta\,\dd x = 0,
\end{equation}
which is the restriction of \eqref{eq:ch-minmove} to the network manifold.

\begin{remark}[Total error and the limits of the temporal estimate]
\label{rem:total-error}
Theorem~\ref{thm:first_order} controls the temporal error of exact minimizing-movements iterates for a fixed spatial discretization. To distinguish the other contributions, let $R_h\phi(t_n)$ denote the exact PDE solution sampled on the comparison grid, let $\phi_h(t_n)$ be the exact solution of a compatible semidiscrete gradient system, let $\phi_h^n$ be its exact time-discrete iterate, and let $\phi_{\theta,h}^{\,n}$ be the sampled network iterate. The triangle inequality gives
\begin{equation}
\begin{aligned}
\norm{R_h\phi(t_n)-\phi_{\theta,h}^{\,n}}_h
\le{}\underbrace{\norm{R_h\phi(t_n)-\phi_h(t_n)}_h}_{\text{spatial discretization}}
+\underbrace{\norm{\phi_h(t_n)-\phi_h^n}_h}_{\text{temporal discretization}}
+\underbrace{\norm{\phi_h^n-\phi_{\theta,h}^{\,n}}_h}_{\text{network approximation}}.
\end{aligned}
\label{eq:total-error}
\end{equation}
Here $\norm{\cdot}_h$ denotes the discrete $L^2$ comparison norm. At fixed spatial resolution, equivalence of finite-dimensional norms transfers the first-order estimate in the flow metric to this comparison norm, so Theorem~\ref{thm:first_order} bounds the middle term for the first-order scheme. 
\end{remark}
\section{Variable-scaled EVNN (VS-EVNN)}

The baseline EVNN used in Section~\ref{sec:nn_approximation} advances the solution by one first-order minimizing-movements solve per time step. Here we implement the multi-stage variational-extrapolation construction. 

\subsection{Variable scaling of the network coordinates}
\label{subsec:variable-scaling}\label{subsec:scaled-network-class}
Phase-field solutions vary across transition layers whose width is set by the coefficient of the gradient term in the energy. Defined directly on the physical domain, a coordinate network must produce order-one changes in the state over regions much narrower than the domain itself. At the same time, its Xavier-initialized first layer generates features that vary on the domain scale. The VS-EVNN solver therefore changes the coordinates presented to the network. For a fixed center \(c\) and a fixed scale \(s>0\), the trial functions are
\begin{equation}
U_\theta(x)=\mathrm{NN}_\theta(\bar{x}),
\qquad
\bar{x}=\frac{x-c}{s},
\label{eq:scaled-coordinate-network}
\end{equation}
where \(\mathrm{NN}_\theta\) is the coordinate network of Section~\ref{sec:nn_approximation}, realized in our experiments as a fully connected residual network. The same \(c\) and \(s\) are used at every time step and stage. The updated-state network \(U_\theta\) plays the role of the network iterate \(\phi_\theta\) of Section~\ref{sec:nn_approximation}. We switch to the symbol \(U\) here because the multi-stage schemes below carry several intermediate states per time step. The implementation takes \(s=\sqrt{\kappa}\), where \(\kappa\) is the coefficient of \(\frac12|\grad\phi|^2\) in the energy. Every test problem in this paper has \(\kappa=\varepsilon^2\), so \(s\) equals the interface parameter \(\varepsilon\). For the one-dimensional Allen--Cahn problem, this gives \(s=\sqrt{10^{-3}}\approx3.2\times10^{-2}\). The remaining values of \(\varepsilon\) are listed with each problem. The center is \(c=0\) on \([-1,1]\) and \(c=(\pi,\pi)\) on \([0,2\pi)^2\).

A change of variables in the energy explains this choice of scale. Consider, in \(d\) spatial dimensions,
\[
\E(\phi)=\int_{\Omega}\left[\frac{\kappa}{2}|\grad_x\phi|^2+F(\phi)\right]\dd x,
\qquad
\phi(x)=\psi(\bar{x}),
\qquad
x=c+s\bar{x}.
\]
Since \(\grad_x\phi=s^{-1}\grad_{\bar{x}}\psi\) and \(\dd x=s^d\,\dd\bar{x}\), we obtain
\begin{equation}
\E(\phi)
=s^d\int_{\bar\Omega}
\left[
\frac{\kappa}{2s^2}\,|\grad_{\bar{x}}\psi|^2
+F(\psi)
\right]\dd\bar{x},
\qquad
\bar\Omega=\frac{\Omega-c}{s}.
\label{eq:scaled-energy}
\end{equation}
Choosing \(s=\sqrt{\kappa}\) gives the gradient term of the transformed integrand a unit coefficient. In the scaled coordinates, a transition layer of physical width proportional to \(\sqrt{\kappa}\) has order-one width, so the network approximates a profile with order-one derivatives.

The solver samples the physical coordinates \(x\), evaluates the physical energy, the movement penalties, and the quadrature weights on the physical grid, and automatic differentiation applies the chain-rule factor \(1/s\) through \eqref{eq:scaled-coordinate-network} wherever \(\grad_x U_\theta\) is required.

\begin{proposition}[Scaled and unscaled coordinate networks coincide]
\label{prop:equiv-coordinate}
Fix \(s>0\) and \(c\in\mathbb{R}^d\), and let the first layer of \(\mathrm{NN}_\theta\) be an affine map \(z\mapsto Wz+b\) with unconstrained weights \(W\) and bias \(b\). Then the trial classes
\[
\mathcal{A}
=\bigl\{x\mapsto \mathrm{NN}_\theta(x)\bigr\}
\qquad\text{and}\qquad
\mathcal{A}_{s,c}
=\Bigl\{x\mapsto \mathrm{NN}_\theta\bigl(\tfrac{x-c}{s}\bigr)\Bigr\}
\]
contain the same functions.
\end{proposition}

\begin{proof}
Given an unscaled first layer \(W_0x+b_0\), the scaled first layer with \(W_s=sW_0\) and \(b_s=b_0+W_0c\) satisfies \(W_s\frac{x-c}{s}+b_s=W_0x+b_0\). Conversely, a scaled first layer \(W_s\frac{x-c}{s}+b_s\) equals the unscaled layer with \(W_0=W_s/s\) and \(b_0=b_s-W_sc/s\). All subsequent layers are shared.
\end{proof}

The implementation pairs the input scaling with a compensating initialization. After Xavier initialization of the base network, the first-layer weights are multiplied by \(s\). The initial preactivations at input magnitudes of \(1/s\) then remain at the scale for which the initialization was designed, avoiding additional saturation caused solely by the input rescaling. During training, the first layer is unconstrained, so Proposition~\ref{prop:equiv-coordinate} still applies to the representable set, but the optimizer follows a different path in parameter space because the reparametrization changes the gradient geometry as seen by L-BFGS.

\subsection{High-order multi-stage schemes} The minimizing-movements update of Section~\ref{subsec:minimizing-movements} is only first-order accurate in time. To raise the temporal order while retaining the variational structure, we adopt the \emph{variational extrapolation} framework of Zaitzeff, Esedo\u{g}lu, and Garikipati~\cite{Zaitzeff2020}, which composes several minimizing-movements stages per time step and penalizes the distance of each stage to all previously computed stages through quadratic movement-limiter terms. Our networks minimize the full energy $\E$ implicitly at every stage, so the fully implicit framework of~\cite{Zaitzeff2020} is the relevant one here. Each stage requires the same type of updated-state solve as the baseline scheme, a minimization of the energy plus quadratic movement terms. VS-EVNN uses the scaled-coordinate state network of Section~\ref{subsec:variable-scaling} for every stage problem.

Definition~\ref{def:k_stage} and Theorem~\ref{thm:multistage_stability} below are imported from~\cite{Zaitzeff2020}, while Theorem~\ref{thm:high_order} restates the formal one-step consistency calculation of~\cite[Proposition~3.1 and \S 4]{Zaitzeff2020} under the smoothness and exact-solve hypotheses listed in the theorem. The additions specific to the present work are the network form of the stage solves in \eqref{eq:network-stage-loss}, and the analysis of the stability guarantee under inexact network minimization in Remark~\ref{rem:stability-network}.

\begin{definition}[$k$-Stage Variational Scheme {\cite{Zaitzeff2020}}]\label{def:k_stage}
Starting from $U_0=\phi^n$, a $k$-stage scheme computes intermediate states $U_1,\ldots,U_k$ via the variational problems
\begin{equation}
U_m
\in
\argmin_{\phi\in\X}
\left(
\E(\phi)
+
\sum_{i=0}^{m-1}
\frac{\gamma_{m,i}}{2\tau}
\norm{\phi-U_i}^2
\right),
\qquad m=1,\ldots,k,
\label{eq:k_stage_problem}
\end{equation}
and the updated solution is taken as $\phi^{n+1}=U_k$. The lower-triangular matrix $\gamma=(\gamma_{m,i})$ collects the scheme coefficients. The column index is zero-based, so row $m$ of a displayed matrix lists $\gamma_{m,0},\ldots,\gamma_{m,m-1}$.
\end{definition}

Throughout this subsection, we write
\begin{equation}
S_m := \sum_{i=0}^{m-1}\gamma_{m,i}, \qquad m=1,\ldots,k,
\label{eq:row-sum}
\end{equation}
for the row sums of $\gamma$. The coefficient matrices used in this work have $S_m>0$, which makes each stage objective coercive in $\phi$ for the energies considered here, since the Ginzburg--Landau energy is bounded below. The quadratic terms jointly grow like $\tfrac{S_m}{2\tau}\norm{\phi}^2$ as $\norm{\phi}\to\infty$. Individual coefficients $\gamma_{m, i}$ may be negative. A stage may carry a negative weight on the distance to one earlier state, as long as the remaining anchors penalize movement strongly enough~\cite{Zaitzeff2020}.

The implementation directly minimizes the multi-anchor stage objective in \eqref{eq:k_stage_problem}, and this network realization of the stage solver is specific to the present work. Stage $m$ is represented by its own scaled-coordinate state network $U_{m,\theta_m}$, warm-started from the preceding stage, and the earlier stages enter the loss only through their values on the quadrature grid, detached from the differentiation graph. The sampled stage loss is
\begin{equation}
\mathcal{L}_{m,h}(\theta_m)
=
\E_h(\widetilde U_{m,\theta_m})
+
\sum_{i=0}^{m-1}
\frac{\gamma_{m,i}}{2\tau}
\norm{\widetilde U_{m,\theta_m}-\widetilde U_i}_{\X,h}^{2}
+
P_{\mathrm{BC},h}(\widetilde U_{m,\theta_m}),
\label{eq:network-stage-loss}
\end{equation}
where $\E_h$ is the sampled energy and $\norm{\cdot}_{\X,h}$ is the discrete counterpart of the flow metric on the quadrature grid, that is, the grid $L^2$ norm for the Allen--Cahn flows and, for the Cahn--Hilliard flows, the mobility-weighted discrete $H^{-1}$ norm $\norm{v}_{\X,h}^2=\frac{1}{M}\norm{v}_{H^{-1},h}^2$, computed by inverting the discrete Laplacian symbol in Fourier space with the zero mode excluded. Here $\widetilde U=U$ in the Allen--Cahn cases while $\widetilde U$ denotes the mean-projected state of Section~\ref{sec:nn_approximation} for Cahn--Hilliard, and the boundary penalty $P_{\mathrm{BC},h}$ is present only in the one-dimensional Allen--Cahn driver, where it is a weight $\lambda_{\mathrm{BC}}$ times the mean squared mismatch between the network trace and the boundary data over the two endpoint nodes. The stage problems are solved in this raw updated-state form. The accepted final-stage network becomes the network of the new time level, so the representation does not grow with the number of time steps.

A central result of~\cite{Zaitzeff2020} is that the multi-stage construction preserves the unconditional energy dissipation of the single-stage scheme established in Proposition~\ref{prop:energy_dissipation}, provided an explicit algebraic condition on $\gamma$ holds, as stated in the following theorem.

\begin{theorem}[Unconditional energy stability {\cite[Theorem~2.1]{Zaitzeff2020}}]
\label{thm:multistage_stability}
Given the coefficients $\gamma_{m,i}$ of Definition~\ref{def:k_stage}, define auxiliary quantities by the backward recursion (from $m=k$ down to $m=1$)
\begin{equation}
\tilde\gamma_{m,i} = \gamma_{m,i} - \sum_{j=m+1}^{k} \tilde\gamma_{j,i}\,\frac{\tilde S_{j,m}}{\tilde S_{j,j}},
\qquad
\tilde S_{j,m} = \sum_{i=0}^{m-1}\tilde\gamma_{j,i},
\label{eq:stability-aux}
\end{equation}
where the sum over $j$ is understood to be empty for $m=k$. If
\begin{equation}
\tilde S_{m,m} > 0 \qquad \text{for } m=1,\ldots,k,
\label{eq:stability-cond}
\end{equation}
then the $k$-stage scheme is unconditionally energy-stable. For every $\tau>0$ and every $n\ge 0$,
\begin{equation}
\E(\phi^{n+1}) \le \E(\phi^n),
\end{equation}
provided each stage minimization in \eqref{eq:k_stage_problem} is solved exactly.
\end{theorem}

Condition \eqref{eq:stability-cond} can be checked for any candidate matrix by running the recursion \eqref{eq:stability-aux} backward. Both coefficient matrices used in this work satisfy it. We verified this numerically from the tabulated entries of the three-stage second-order matrix and the six-stage third-order matrix in Table~\ref{tab:high_order_coefficients}.

Separate algebraic conditions on $\gamma$ govern the order of accuracy, analogous to Runge--Kutta order conditions. Taylor expansion of the stage Euler--Lagrange equations shows that each stage admits an expansion of the form $U_m=U_0-\beta_{1,m}\,\tau\,\grad\E(U_0)+O(\tau^2)$, whose coefficients $\beta_{j,m}$ obey the recursion (with $\beta_{j,0}=0$ and the row sums $S_m$ of \eqref{eq:row-sum})
\begin{subequations}\label{eq:order-recursion}
\begin{equation}
\beta_{1,m}=\frac{1}{S_m}\Bigl[1+\sum_{i=1}^{m-1}\gamma_{m,i}\beta_{1,i}\Bigr],
\qquad
\beta_{2,m}=\frac{1}{S_m}\Bigl[\beta_{1,m}+\sum_{i=1}^{m-1}\gamma_{m,i}\beta_{2,i}\Bigr],
\label{eq:order-recursion-a}
\end{equation}
\begin{equation}
\beta_{3,m}=\frac{1}{S_m}\Bigl[\beta_{2,m}+\sum_{i=1}^{m-1}\gamma_{m,i}\beta_{3,i}\Bigr],
\qquad
\beta_{4,m}=\frac{1}{S_m}\Bigl[\frac{\beta_{1,m}^2}{2}+\sum_{i=1}^{m-1}\gamma_{m,i}\beta_{4,i}\Bigr].
\label{eq:order-recursion-b}
\end{equation}
\end{subequations}
The scheme is first-order accurate if and only if $\beta_{1,k}=1$. It is second-order accurate if, in addition, $\beta_{2,k}=1/2$. It is third-order accurate if, further, $\beta_{3,k}=\beta_{4,k}=1/6$.

\begin{theorem}[Formal one-step consistency for exact stage solves {\cite[Proposition~3.1 and \S 4]{Zaitzeff2020}}]
\label{thm:high_order}
Let $U_0=\phn$, let $\phi$ denote the exact solution of the gradient flow with $\phi(t_n)=U_0$, and assume that $\phi$ and the energy $\E$ are sufficiently smooth and that each stage minimization in Definition~\ref{def:k_stage} is solved exactly along a smooth, locally unique stage branch. If the coefficient matrix $\gamma=(\gamma_{m, i})$ satisfies the stability condition \eqref{eq:stability-cond} and the order conditions above through second order, then the one-step error satisfies $\norm{U_k-\phi(t_n+\tau)}=O(\tau^{3})$ with $k=3$ stages. If $\gamma$ satisfies the order conditions through third order, then $\norm{U_k-\phi(t_n+\tau)}=O(\tau^{4})$ with $k=6$ stages. The corresponding exact-stage schemes therefore have formal temporal orders of 2 and 3, respectively.
\end{theorem}

No unconditionally energy-stable two-stage second-order scheme of the form \eqref{eq:k_stage_problem} exists~\cite{Zaitzeff2020}, so $k=3$ is minimal. Theorem~\ref{thm:high_order} is stated in terms of unnormalized one-step defects, obtained by Taylor expansion of the stage Euler--Lagrange equations~\cite{Zaitzeff2020}, so an order-$p$ scheme carries a defect of $O(\tau^{p+1})$ over a single step. The corresponding coefficient matrices are given in Table~\ref{tab:high_order_coefficients}.

\begin{table}[H]
\centering
\caption{Coefficient matrices $\gamma$ for the first-, second-, and third-order multi-stage schemes~\cite{Zaitzeff2020}.}
\label{tab:high_order_coefficients}
\begin{tabular}{ccc}
\toprule
Order & Stages & $\gamma$ matrix \\
\midrule
1 & 1 & $[1]$ \\
2 & 3 & $\begin{bmatrix} 5 & 0 & 0 \\ -2 & 6 & 0 \\ -2 & 3/14 & 44/7 \end{bmatrix}$ \\
3 & 6 & Eq.~\eqref{eq:gamma_decimal} \\
\bottomrule
\end{tabular}
\end{table}

For the third-order six-stage scheme, the update from $\phi^n$ to $\phi^{n+1}$ uses Definition~\ref{def:k_stage} with time step $\tau$ and the coefficient matrix $\gamma$ of~\cite{Zaitzeff2020}. In our single-precision implementation, we evaluate the coefficients in decimal form, accurate to the digits shown,
\begin{equation}
\label{eq:gamma_decimal}
\gamma
\approx
\begin{bmatrix}
11.1666667 & 0 & 0 & 0 & 0 & 0 \\
-7.5 & 19.4285714 & 0 & 0 & 0 & 0 \\
-1.05 & -4.75 & 13.9761905 & 0 & 0 & 0 \\
1.8 & 0.047619 & -7.8333333 & 13.8 & 0 & 0 \\
6.2 & -7.1666667 & -1.3333333 & 1.625 & 11.5238095 & 0 \\
-2.8333333 & 4.6875 & 2.4577341 & -11.551686 & 6.6801207 & 11.9455237
\end{bmatrix}.
\end{equation}
The exact corrected coefficients satisfy the third-order conditions. The displayed decimals satisfy them only up to coefficient-rounding errors. We retain the displayed precision to limit these errors.

\begin{remark}[Energy stability in the network setting]
\label{rem:stability-network}
Theorem~\ref{thm:multistage_stability} controls the energy only at full time steps; intermediate stages need not decrease it monotonically. In the sampled network implementation, define
\[
\mathcal G_h(U):=\E_h(U)+P_{\mathrm{BC},h}(U),\qquad
J_{m,h}(U):=\mathcal G_h(U)+\sum_{i=0}^{m-1}\frac{\gamma_{m,i}}{2\tau}\norm{U-U_i}_{\X,h}^2,
\]
where all states are mean-projected in the Cahn--Hilliard cases. Provided the same admissible trial class, quadrature, metric, and boundary penalty are used at every stage, the preceding stage is an admissible competitor. If the computed stages satisfy
\[
J_{m,h}(U_m)\le J_{m,h}(U_{m-1}),\qquad m=1,\ldots,k,
\]
the algebraic argument of~\cite[Theorem~2.1]{Zaitzeff2020} gives $\mathcal G_h(U_k)\le\mathcal G_h(U_0)$ without requiring global minimization. For the single-stage scheme, the required check is
\[
\mathcal G_h(U_1)+\frac{1}{2\tau}\norm{U_1-U_0}_{\X,h}^2\le\mathcal G_h(U_0).
\]
When the boundary penalty is absent or vanishes identically, $\mathcal G_h = \ E_h$. 
\end{remark}

\subsection{Adaptive time stepping}
\label{subsec:time-adaptive}
We adjust the time step using a heuristic relative-change controller to improve efficiency. The step shrinks during rapidly evolving periods, where accuracy is at risk, and grows during slowly evolving periods, such as coarsening, where computation can be saved.

After the $(n+1)$-th time step is computed, we measure the relative change of the solution by
\begin{equation}
e_{n+1} = \frac{\norm{\phi^{n+1} - \phi^{n}}_{L^2}}{\norm{\phi^{n+1}}_{L^2}}.
\end{equation}
Here both norms are the discrete $L^2$ norms on the quadrature grid, for the $L^2$ and the $H^{-1}$ flows alike. The quantity $e_{n+1}$ is large when the solution changes rapidly and small once the dynamics slow.

The value of $e_{n+1}$ is then compared with a prescribed tolerance $\mathrm{tol}$ to decide whether the step is accepted. If $e_{n+1} \ge \mathrm{tol}$ and the current step size $\tau_n > \tau_{\min}$, the step is rejected, the step size is reduced, and the solution is recomputed at the current time level. If $e_{n+1} < \mathrm{tol}$, the result is accepted and the step size is updated for the next step. If the step size has already reached the prescribed minimum $\tau_{\min}$, no further reduction is made, and the step is accepted.

The updated time step is given by
   \begin{equation}
\tau_{\mathrm{ada}} = \rho \sqrt{\frac{\mathrm{tol}}{\max(e_{n+1},\delta_e)}} \cdot \tau_{n},
\label{eq:tau_ada}
\end{equation}
where $\tau_{n}$ is the current time step, $\rho\in(0,1)$ is a safety factor, and $\mathrm{tol}$ is the reference tolerance. We use the square-root exponent for all adaptive runs. When $e_{n+1}$ is well below the tolerance, the factor $\rho\sqrt{\mathrm{tol}/e_{n+1}}$ exceeds $1$, and the step size increases. As $e_{n+1}$ approaches the tolerance, the factor tends toward $\rho < 1$. The accepted-step update is in equilibrium at $e_{n+1}=\rho^2\,\mathrm{tol}$, where the factor equals one, so the controller regulates the activity indicator around $\rho^2\,\mathrm{tol}$ rather than $\mathrm{tol}$ itself.

\section{Numerical examples}
\label{sec:numerics}

We organize the experiments by question rather than by equation. The questions concern comparisons with the EVNN baseline and a space--time PINN, the accuracy of the multi-stage schemes relative to the non-temporal error floor, and adaptive time stepping.

\subsection{Test problems and implementation setup}

\subsubsection{Test problems}
\label{subsec:test-problems}

\paragraph{One-dimensional Allen--Cahn equation.} We consider the one-dimensional Allen--Cahn equation as the \(L^2\) gradient flow
\begin{equation}
\partial_t \phi = -\frac{\delta \E}{\delta \phi},
\end{equation}
associated with the free energy functional $ \E(\phi) = \int_{-1}^{1} \left( \frac{\varepsilon^2}{2} (\partial_x\phi)^2 + \frac{5}{4}(\phi^2 - 1)^2 \right) \dd x, $ where the interface parameter $\varepsilon$ satisfies $\varepsilon^2 = 0.001$. The double-well prefactor here is $5/4$ rather than the $1/4$ of \eqref{eq:ch-energy}, but the scaling argument of Section~\ref{subsec:variable-scaling} applies to a generic double-well $F$. The corresponding partial differential equation reads
\begin{subequations}\label{eq:ac1d-problem}
\begin{equation}
\partial_t \phi - 0.001 \, \partial_x^2 \phi + 5\phi^3 - 5\phi = 0,
\qquad x \in [-1, 1], \; t \in [0, 1],
\end{equation}
with Dirichlet boundary conditions
\begin{equation}
\phi(-1, t) = \phi(1, t) = -1, \qquad t \in [0, 1],
\end{equation}
and the initial condition
\begin{equation}
\phi(x, 0) = x^2 \cos(\pi x), \qquad x \in [-1, 1].
\end{equation}
\end{subequations}

\paragraph{Two-dimensional Allen--Cahn equation.} We consider the Allen--Cahn equation as the $L^2$ gradient flow of the Ginzburg--Landau free energy:
\begin{equation}
\partial_t \phi
=
\lambda
\left(
\varepsilon^2 \Delta \phi - \phi^3 + \phi
\right),
\qquad
(x,y)\in \Omega=[0,2\pi)^2 .
\end{equation}
We impose periodic boundary conditions. The free energy functional is given by
\begin{equation}
\E(\phi)
=
\int_{\Omega}
\left(
\frac{\varepsilon^2}{2}|\grad \phi|^2
+
\frac{1}{4}(\phi^2-1)^2
\right)
\dd x\,\dd y .
\end{equation}
We take the interface parameter $\varepsilon=0.08$ and the relaxation parameter $\lambda=1$. 

\paragraph{Star-shaped interface.} The initial condition is a star-shaped profile:
\begin{equation}
\phi_0(x,y)
=
\tanh
\left(
\frac{R(\vartheta)-r}{\varepsilon\sqrt{2}}
\right),
\qquad
R(\vartheta)=R_0+A\cos(6\vartheta),
\qquad
r=\sqrt{(x-\pi)^2+(y-\pi)^2},
\end{equation}
where $\vartheta$ is the polar angle about the domain center $(\pi,\pi)$, with $R_0=1.5$ and $A=0.4$, so the interface of $\phi_0$ is a six-pointed star.

\paragraph{Randomly placed circles.} The initial condition is generated by placing $N_c$ circles with random centers and radii within the domain. The phase field associated with the $i$th circle is
\begin{equation}
q_i(x,y) = \tanh\!\left( \frac{R_i - r_i(x,y)}{\varepsilon\sqrt{2}} \right),
\qquad i=1,\ldots,N_c,
\end{equation}
where $R_i$ is its radius and $r_i(x,y)=\sqrt{(x-x_i)^2+(y-y_i)^2}$ is the Euclidean (not toroidal) distance to its center $(x_i,y_i)$. The radii are uniformly sampled from $[0.5, 0.8]$. The centers are uniformly distributed in $[L/4, 3L/4]^2$ with $L=2\pi$, so every circle keeps a margin of at least $\pi/2 - 0.8 \approx 0.77$ from the boundary. This margin is much larger than the interface parameter $\varepsilon=0.08$, and the initial interfaces stay strictly inside the domain. The resulting initial condition is the pointwise maximum $\phi_0(x,y)=\max_{1\le i\le N_c} q_i(x,y)$. Unless stated otherwise, we use $N_c = 5$ circles.

\paragraph{One-dimensional Cahn--Hilliard equation.} We take the one-dimensional Cahn--Hilliard equation on $\Omega=[-1,1]$ with periodic boundary conditions,
\begin{equation}
\partial_t\phi = M\,\partial_{xx}\mu,
\qquad
\mu = -\varepsilon^2\partial_{xx}\phi + \phi^3 - \phi,
\end{equation}
which is the $H^{-1}$ gradient flow of the Ginzburg--Landau energy \eqref{eq:ch-energy}, discretized by the VS-EVNN scheme in the $H^{-1}$ metric of Section~\ref{subsec:ch-metric}. The initial condition is
\begin{equation}
\phi(x,0)=\cos(\pi x), \qquad x\in[-1,1],
\end{equation}
whose mean is zero, so the dynamics evolve on the mass manifold $\mathcal{M}_{m_0}$ with $m_0=0$. Two parameter regimes are tested. The coarsening regime uses mobility $M=2\times10^{-4}$ and interface parameter $\varepsilon=0.01$ and runs to $T_{\text{final}}=1000$. The sharp-interface regime uses $M=1$ and the much smaller interface parameter $\varepsilon=0.002$ and is resolved on a finer grid $N=4096$.

\paragraph{Two-dimensional Cahn--Hilliard equation.} We solve the Cahn--Hilliard equation on the periodic domain $\Omega=[0,2\pi)^2$,
\begin{equation}
\partial_t\phi = M\Delta\mu,
\qquad
\mu = -\varepsilon^2\Delta\phi + \phi^3 - \phi,
\end{equation}
with constant mobility $M=0.01$ and interface parameter $\varepsilon=0.08$. The energy is the Ginzburg--Landau functional \eqref{eq:ch-energy}. The initial condition is the pointwise maximum of $N_c=5$ circular phase-field profiles with randomly generated centers and radii using seed 42, exactly as in the randomly-placed-circles construction above, and the final time is $T_{\text{final}}=100$.

\subsubsection{Implementation details and reproducibility}
\label{subsec:implementation}

All EVNN and VS-EVNN drivers use a width-64 ResNet with three two-layer residual blocks, tanh activation, and Xavier-uniform weight initialization. The input and output dimensions are $1\to1$ in one dimension and $2\to1$ in two dimensions. The VS-EVNN drivers evaluate the coordinates through \eqref{eq:scaled-coordinate-network}, and after the Xavier initialization, the first-layer weights are multiplied by $s$ to compensate for the $1/s$ input magnitude, following Section~\ref{subsec:scaled-network-class}. The architecture is therefore matched between the EVNN and VS-EVNN comparison paths.

\paragraph{Stage representation and inner solvers.} Every stage is an updated-state network warm-started from the preceding stage, with the first stage warm-started from the accepted network of the current time level. Earlier stages are evaluated on the common quadrature grid and detached. The loss is the multi-anchor objective \eqref{eq:network-stage-loss}, so the representation does not grow with the number of time steps. The inner solver is L-BFGS with a strong-Wolfe line search, stopping at the iteration cap or once the gradient or parameter change falls below the configured tolerances. The two-dimensional Allen--Cahn drivers allow $500$ to $1000$ iterations, and the Cahn--Hilliard drivers precede L-BFGS with an Adam warm-up, namely $500$ Adam iterations at learning rate $10^{-4}$ followed by at most $500$ L-BFGS iterations in the fixed-step drivers, and $100$ Adam iterations followed by at most $100$ L-BFGS iterations in the adaptive two-dimensional driver. The mean-value correction of Section~\ref{sec:nn_approximation} enforces the mass constraint in the Cahn--Hilliard drivers.

\paragraph{Quadrature and differentiation.} The one-dimensional Allen--Cahn drivers use an endpoint-inclusive uniform grid on $[-1,1]$ with spacing $\Delta x = 2/(N-1)$, compute spatial derivatives by automatic differentiation through the coordinate network, and evaluate the energy and movement terms by Riemann sums of the form $\sum(\cdot)\,\Delta x$. The two-dimensional Allen--Cahn drivers use automatic differentiation on endpoint-excluded $128\times128$ grids on $[0,2\pi)^2$. Both Cahn--Hilliard families use endpoint-excluded periodic grids ($N=1024$ for $\varepsilon=0.01$ and $N=4096$ for $\varepsilon=0.002$ in one dimension, $128\times128$ in two dimensions) and evaluate the energy and the $H^{-1}$ movement terms with second-order finite-difference Laplacian symbols and FFTs rather than network derivatives. In every case, the grid specifies the resolution of the energy quadrature and collocation, not a count of mesh degrees of freedom.

All four reference solutions are computed by semi-implicit (IMEX) schemes of first order in time, treating the stiff linear operator implicitly and the cubic nonlinearity explicitly. Their accuracy comes from over-resolving the time step rather than from raising the order. The one-dimensional Allen--Cahn reference discretizes the diffusion term using the three-point Laplacian, solves the resulting tridiagonal system for the interior unknowns, treats the term $-5(u^3-u)$ explicitly, and holds the Dirichlet data at $ x=-1$. It runs on $2049$ endpoint-inclusive points with time step $10^{-4}$. The one-dimensional Cahn--Hilliard reference advances $(I+\tau\,M\varepsilon^2 D_4)\phi^{n+1}=\phi^n+\tau\,M D_2(\phi^3-\phi)$ on $1024$ endpoint-excluded points with default time step $10^{-5}$, and restores the mean after every step. Here, $D_2$ is the periodic three-point Laplacian, and $ D_4 = D_2^2$. The two-dimensional references are Fourier IMEX kernels on the $128\times128$ grid. The Allen--Cahn kernel solves $(1-\tau\widehat{L})\widehat\phi^{\,n+1}=\widehat{\bigl(\phi^n-\tau\lambda(\phi^n)^3\bigr)}$ with $\widehat{L}=\lambda(1-\varepsilon^2|k|^2)$, and the Cahn--Hilliard kernel solves $(1+\tau M\varepsilon^2|k|^4)\widehat\phi^{\,n+1}=\widehat\phi^{\,n}-\tau M|k|^2\,\widehat{\bigl((\phi^n)^3-\phi^n\bigr)}$ with the mean likewise restored after every step. Here $\widehat{\,\cdot\,}$ is the discrete Fourier transform. The two-dimensional Cahn--Hilliard reference uses the default time step $10^{-2}$.

\paragraph{Precision, hardware, seeds, and initial-condition fitting.} The neural parameters and sampled tensors are single-precision (float32), while the NumPy and SciPy reference and FFT paths use double precision. All experiments are implemented in PyTorch, and the neural runs are reported on an NVIDIA RTX 3090 GPU. Unless stated otherwise, the reported results use seed 3407 for the one-dimensional Allen--Cahn experiments and seed 42 for all other experiments. Table~\ref{tab:2d-seeds} reports the additional seeds 1--5. For the initial-condition fitting, the network is pretrained with Adam for at most 5000 epochs per round, with the learning rate decaying by a factor of 0.7 every 1000 steps, and then refined with L-BFGS. If the fitting target is missed, up to three additional rounds (two in the two-dimensional Cahn--Hilliard drivers) are run at a tenfold-reduced rate. The one-dimensional Allen--Cahn drivers start from learning rate $10^{-2}$ and the other drivers from $10^{-3}$. The $L_\infty$ fitting targets range from $10^{-6}$ for the one-dimensional VS-EVNN drivers to $10^{-2}$ for the two-dimensional Cahn--Hilliard drivers, and each driver issues a warning if its target is not met.

\paragraph{PINN baseline.} The PINN baseline is a monolithic space--time network taking $(x,t)$ as input. The one-dimensional PINN drivers use a width-128 network with five residual blocks and tanh activation, wider and deeper than the VS-EVNN stage networks, and an unweighted sum of the PDE, initial-condition, and boundary losses. The drivers use $10\,000$ interior, $512$ initial, and $256$ boundary points, drawn once rather than resampled. The interior points are Latin-hypercube samples for the Allen--Cahn driver and uniform random samples for the Cahn--Hilliard driver, and the initial and boundary points are linearly spaced. The Allen--Cahn path trains with Adam for $50\,000$ iterations and the Cahn--Hilliard path for $20\,000$, with no causal weighting, adaptive sampling, time marching, or L-BFGS refinement, and the Cahn--Hilliard boundary loss matches only part of the periodic closure of the fourth-order strong form.

\subsection{Spatial coordinate scaling versus the EVNN baseline} We first test whether the scaled-coordinate representation (VS-EVNN) improves accuracy over the EVNN baseline at a fixed time step and spatial resolution. We compare the two methods on the one-dimensional Allen--Cahn and Cahn--Hilliard problems, and then show a case where the unscaled-coordinate network (EVNN) fails outright.

\subsubsection{One-dimensional Allen--Cahn equation} We compare the first-order VS-EVNN and EVNN schemes on the one-dimensional Allen--Cahn problem of Section~\ref{subsec:test-problems}. Figure~\ref{fig:1d-spatial-compare} shows the computed solutions of the two methods on a 1024-point grid at the common time step $\tau =6.25\times 10^{-3}$, and Table~\ref{tab:1d-evnn-vs} reports their relative $L^2$ errors at $t=1$ for several spatial grids and time steps. For a fixed time step and spatial grid, VS-EVNN is far more accurate than the baseline EVNN under the reported optimizer settings.

\begin{figure}[H]
\centering
\begin{subfigure}[t]{0.8\textwidth}
\centering
\includegraphics[width=\linewidth]{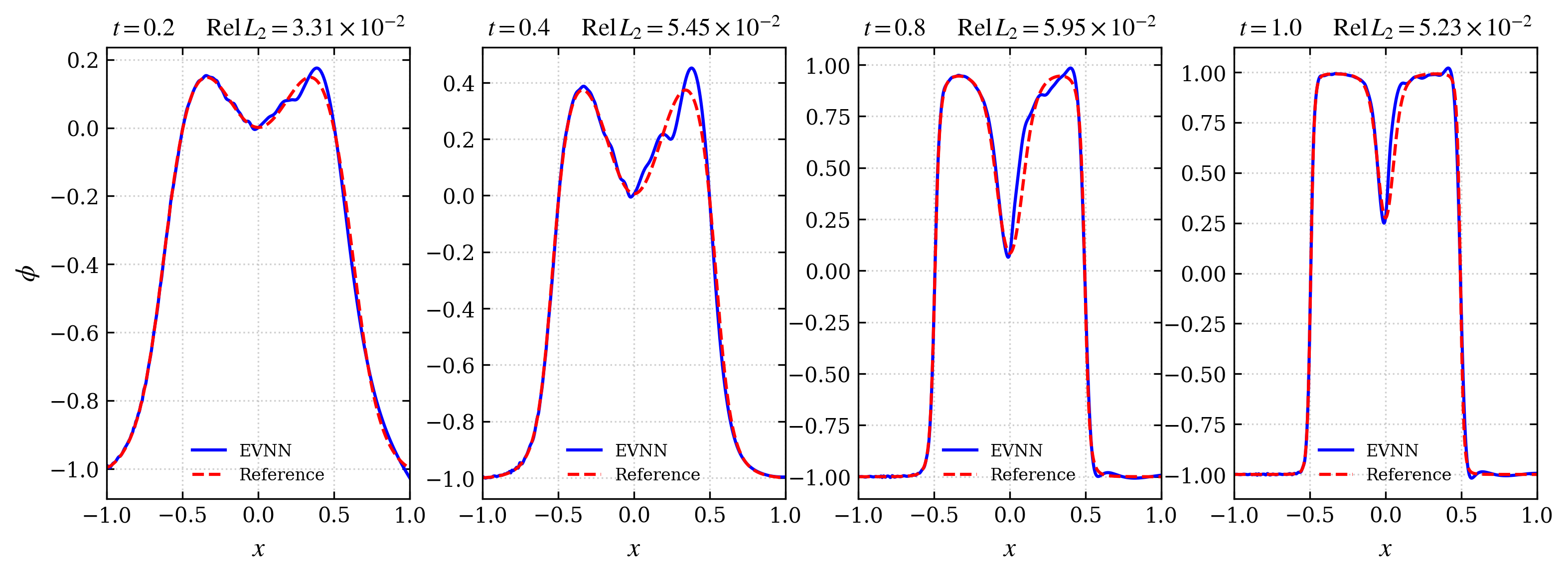}
\caption{EVNN, grid 1024, $\tau=6.25 \times 10^{-3}$}
\end{subfigure}

\begin{subfigure}[t]{0.8\textwidth}
\centering
\includegraphics[width=\linewidth]{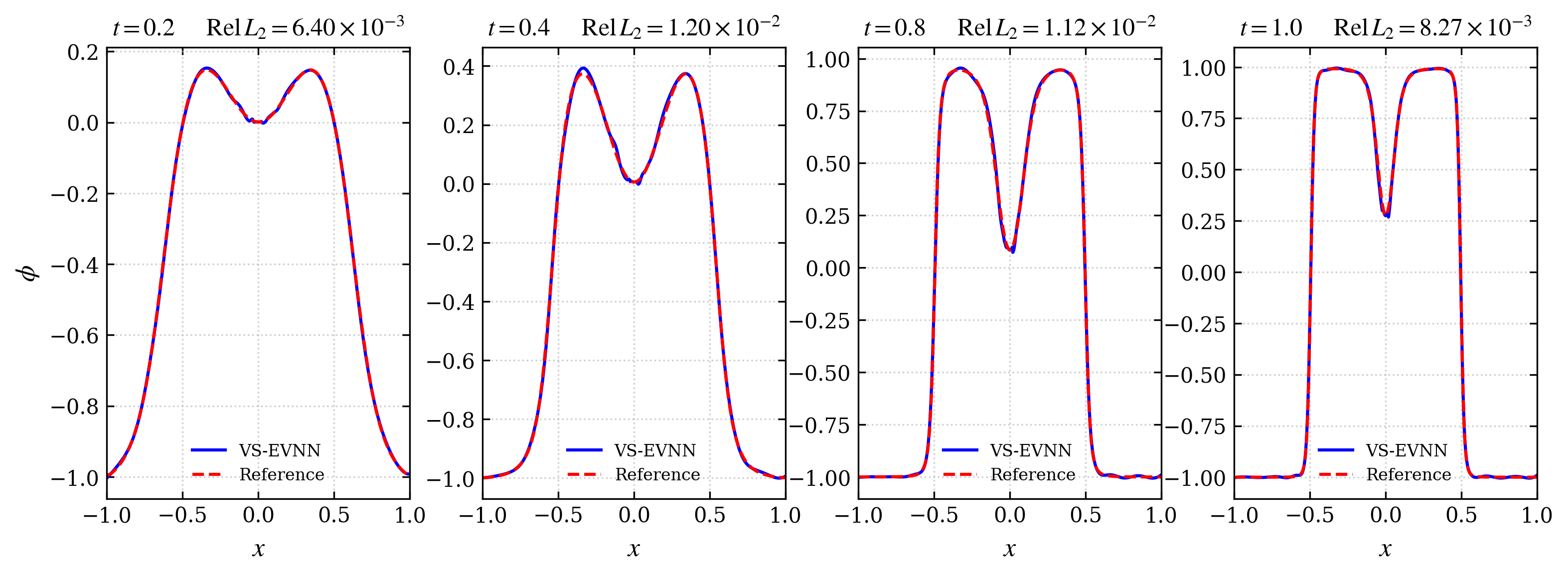}
\caption{VS-EVNN, grid 1024, $\tau=6.25 \times 10^{-3}$}
\end{subfigure}
\caption{One-dimensional Allen--Cahn equation. First-order EVNN versus VS-EVNN on a 1024-point grid at the same time step $\tau=6.25 \times 10^{-3}$. Each panel shows the computed solution (solid) and the reference (dashed) at $t=0.2, 0.4, 0.8, 1.0$.}
\label{fig:1d-spatial-compare}
\end{figure}
\begin{table}[H]
\centering
\caption{Relative $L^2$ errors of the first-order EVNN and VS-EVNN schemes at $t=1$ for the one-dimensional Allen--Cahn problem, under different spatial grids and time steps. The last column is the ratio of the EVNN error to the VS-EVNN error.}
\label{tab:1d-evnn-vs}
\begin{tabular}{c c c c c}
\toprule
 & & \multicolumn{2}{c}{Relative $L^2$ error} & \\
\cmidrule(lr){3-4}
Grid $N$ & $\tau$ & EVNN & VS-EVNN & EVNN/VS-EVNN \\
\midrule
512  & $6.25\times 10^{-3}$ & $7.14\times10^{-2}$ & $1.57\times10^{-2}$ & $4.55$ \\
1024 & $6.25\times 10^{-3}$ & $5.23\times10^{-2}$ & $8.27\times10^{-3}$ & $6.32$ \\
2048 & $6.25\times 10^{-3}$ & $3.44\times10^{-2}$ & $9.03\times10^{-3}$ & $3.81$ \\
1024 & $1.25\times 10^{-2}$ & $6.49\times10^{-2}$ & $1.71\times10^{-2}$ & $3.80$ \\
\bottomrule
\end{tabular}
\end{table}

\subsubsection{One-dimensional Cahn--Hilliard equation near the steady state} The conserved setting repeats the comparison in the $H^{-1}$ metric, using the coarsening regime of Section~\ref{subsec:test-problems}. Table~\ref{tab:1d-ch-evnn-vs} reports the relative $L^2$ errors of the two methods near the steady state at $T=1000$ for the steps $\tau=50$ and $\tau=25$. For the second- and third-order schemes, VS-EVNN is substantially more accurate at the same step. At $\tau=25$, the third-order EVNN error ($3.37\times10^{-3}$) is more than four times the VS-EVNN error ($7.31\times10^{-4}$). The advantage of coordinate scaling is evident in both metrics tested under the reported optimizer settings. It is most visible near the steady state, where the solution consists of well-formed interface profiles at the scale the scaled network represents efficiently.

\begin{table}[H]
\centering
\caption{Relative $L^2$ errors of the first-, second-, and third-order EVNN and VS-EVNN schemes at steps $\tau=50$ and $\tau=25$ for the one-dimensional Cahn--Hilliard problem ($M=2\times10^{-4}$, $\varepsilon=0.01$), measured at $ T=1000$.}
\label{tab:1d-ch-evnn-vs}
\begin{tabular}{c c c c c}
\toprule
 & \multicolumn{2}{c}{$\tau=50$} & \multicolumn{2}{c}{$\tau=25$} \\
\cmidrule(lr){2-3}\cmidrule(lr){4-5}
Order & EVNN & VS-EVNN & EVNN & VS-EVNN \\
\midrule
1 & $1.71\times10^{-3}$ & $1.69\times10^{-3}$ & $1.13\times10^{-3}$ & $9.29\times10^{-4}$ \\
2 & $9.37\times10^{-4}$ & $6.25\times10^{-4}$ & $2.65\times10^{-3}$ & $6.81\times10^{-4}$ \\
3 & $1.39\times10^{-3}$ & $5.00\times10^{-4}$ & $3.37\times10^{-3}$ & $7.31\times10^{-4}$ \\
\bottomrule
\end{tabular}
\end{table}

\subsubsection{An EVNN optimization failure mode} Table~\ref{tab:1d-ch-evnn-vs} reports accuracy at steps where both optimizations succeed. The unscaled optimization can also fail outright. At $\tau=12.5$, the smallest step tested for this problem, the first-order EVNN reaches an unphysical steady state, with order-one solution errors and a spurious extra interface near $x=0.5$ that persists to the final time, while the first-order VS-EVNN at the same step tracks the reference correctly (Figure~\ref{fig:1d-ch-tau125}). The EVNN energy correspondingly stalls above the reference decay.

\begin{figure}[H]
\centering
\begin{subfigure}[t]{0.85\textwidth}
\centering
\includegraphics[width=\linewidth]{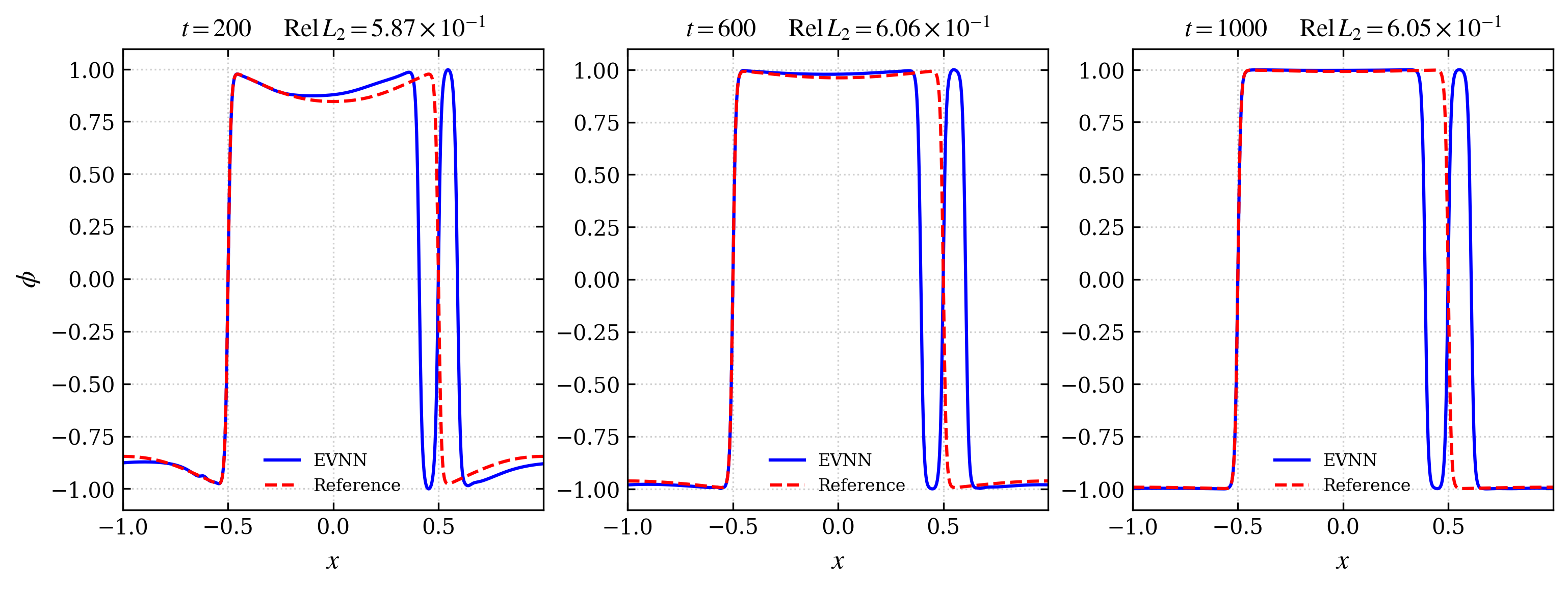}
\caption{First-order EVNN}
\end{subfigure}

\begin{subfigure}[t]{0.85\textwidth}
\centering
\includegraphics[width=\linewidth]{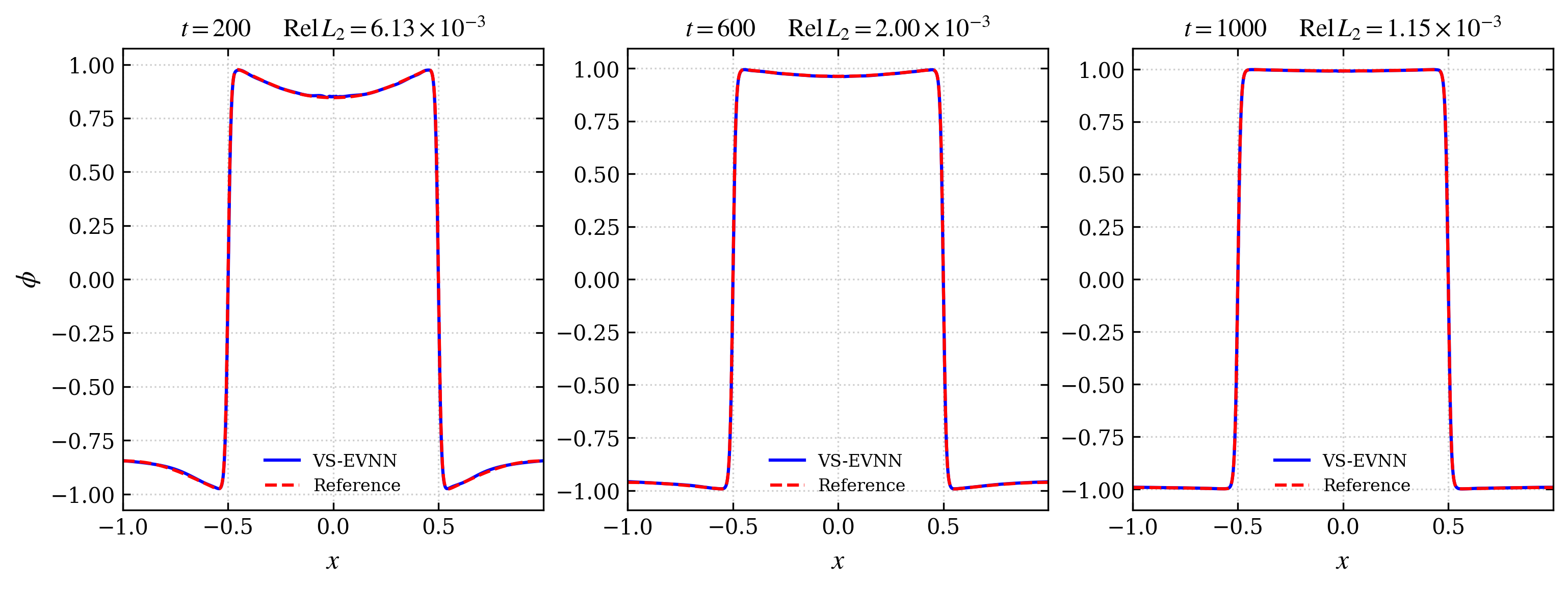}
\caption{First-order VS-EVNN}
\end{subfigure}
\caption{One-dimensional Cahn--Hilliard equation ($\varepsilon=0.01$, $N=1024$, $\tau=12.5$). Computed solutions (solid) and reference solutions (dashed) for the first-order EVNN and VS-EVNN. The EVNN optimization converges to an unphysical steady state, while VS-EVNN tracks the reference.}
\label{fig:1d-ch-tau125}
\end{figure}

\subsection{Comparison with a space--time PINN baseline}

For the one-dimensional Allen--Cahn problem, Figures~\ref{fig:1d-l2} and~\ref{fig:1d-energy-l2} track the relative $L^2$ error of the solution and the relative energy error over time (spatial grid 1024). The VS-EVNN schemes are one to three orders of magnitude more accurate than this basic PINN baseline. The PINN error grows steadily and reaches $O(1)$ by $t=1$, while the sequential variational time stepping keeps the error essentially flat (Figure~\ref{fig:1d-l2}). The benefit of higher order is mainly evident in the energy (Figure~\ref{fig:1d-energy-l2}). The third-order scheme keeps the relative energy error near $10^{-3}$ throughout. The second-order scheme drifts slightly upward, and the first-order scheme fluctuates around $10^{-2}$.

\begin{figure}[H]
\centering
\begin{subfigure}[t]{0.48\textwidth}
\centering
\includegraphics[width=\linewidth]{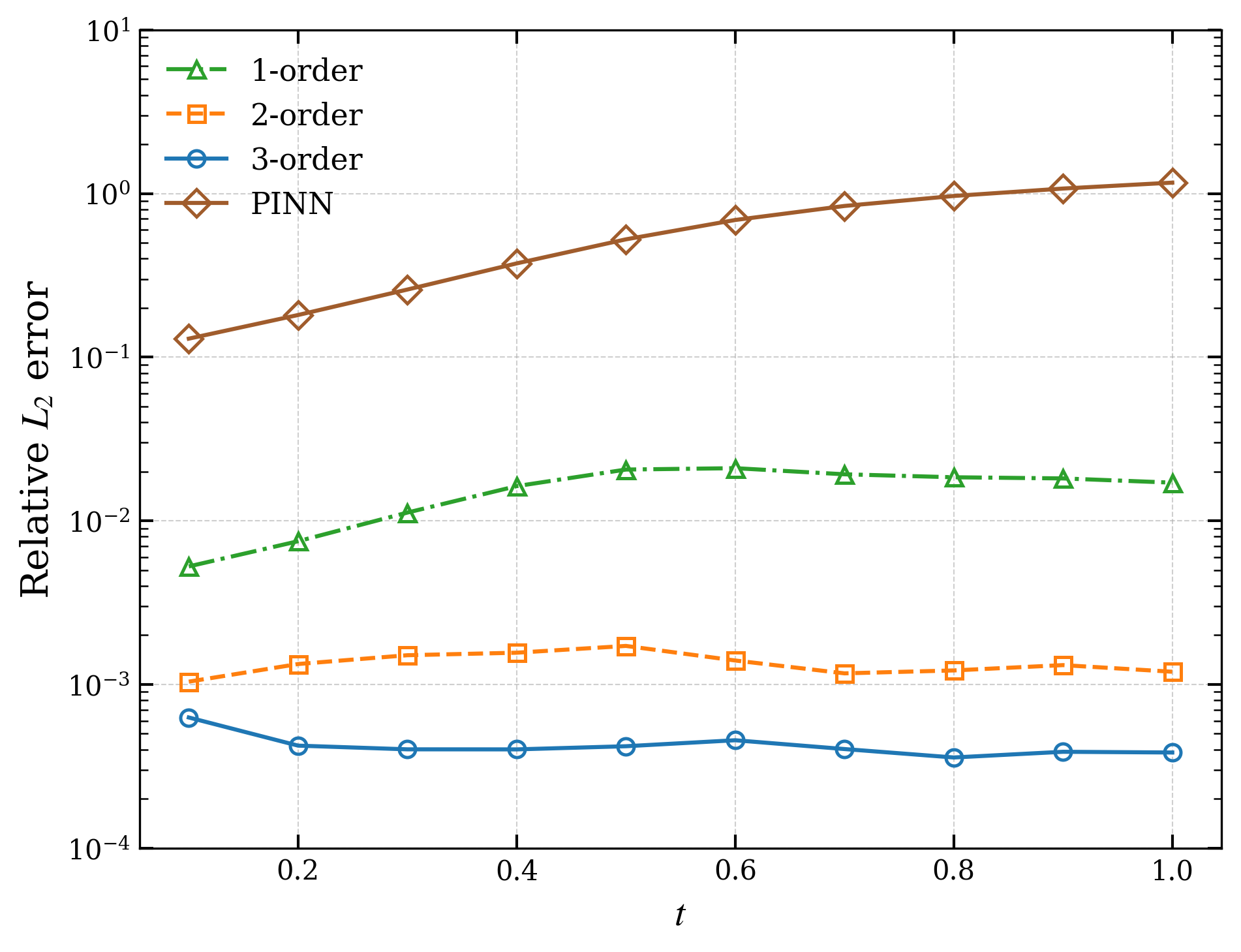}
\caption{Relative $L^2$ error}
\label{fig:1d-l2}
\end{subfigure}\hfill
\begin{subfigure}[t]{0.48\textwidth}
\centering
\includegraphics[width=\linewidth]{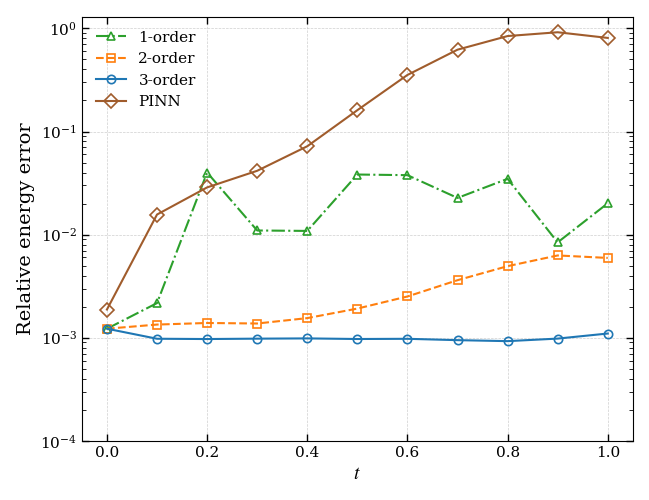}
\caption{Relative energy error}
\label{fig:1d-energy-l2}
\end{subfigure}
\caption{One-dimensional Allen--Cahn equation (spatial grid 1024). (a) Relative $L^2$ error of the solution and (b) relative energy error over time for the first-, second-, and third-order VS-EVNN schemes and the PINN baseline.}
\end{figure}

The Cahn--Hilliard test shows the same ranking of the methods. Figure~\ref{fig:1d-ch-pinn} compares the third-order EVNN and VS-EVNN with the PINN baseline at $\tau=25$ for the one-dimensional Cahn--Hilliard problem in the coarsening regime. Both variational schemes are two to three orders of magnitude more accurate than this basic PINN baseline, whose error saturates near order one almost immediately. The EVNN error levels off at a few times $10^{-3}$, while the VS-EVNN error continues to decrease through the final time, and from $t\approx200$ onward VS-EVNN is consistently the more accurate of the two.

\begin{figure}[H]
\centering
\includegraphics[width=0.48\textwidth]{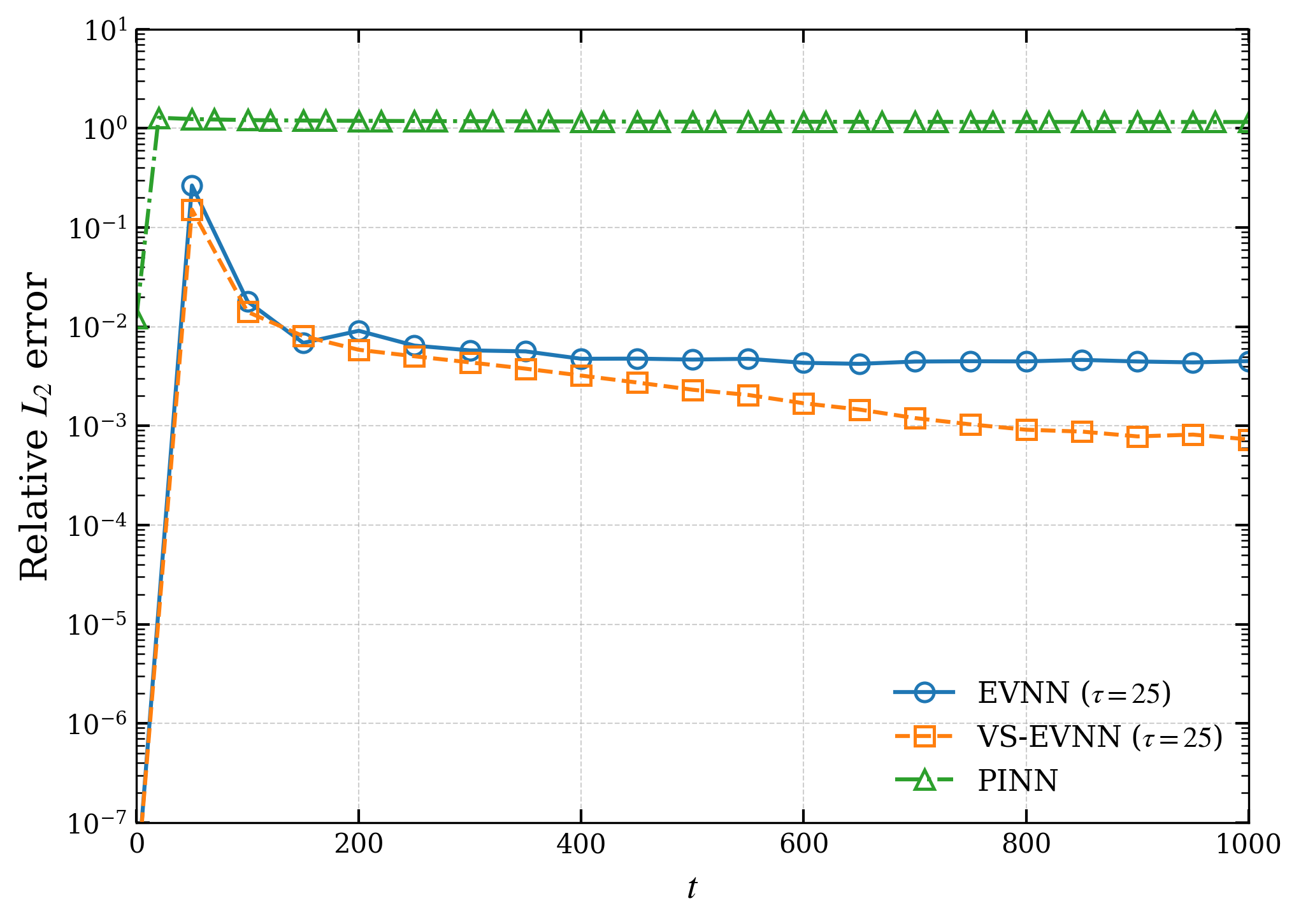}
\caption{One-dimensional Cahn--Hilliard equation ($\varepsilon=0.01$). Relative $L^2$ errors of the third-order EVNN, the third-order VS-EVNN, and the PINN baseline. The variational schemes use $\tau=25$.}
\label{fig:1d-ch-pinn}
\end{figure}

\subsection{Higher-order accuracy and the non-temporal error floor} The experiments in this subsection assess the accuracy gains of the second- and third-order multi-stage schemes at a fixed step. Before the floor is reached, these gains are substantial. For the one-dimensional Allen--Cahn test, the third-order scheme reduces the fixed-step relative $L^2$ error of the first-order scheme by more than one order of magnitude. In the long-horizon runs of Section~\ref{subsec:adaptive-results}, the third-order scheme also maintains the best agreement with the reference energy.

\subsubsection{Temporal convergence for the one-dimensional Allen--Cahn equation} The second- and third-order VS-EVNN solutions on a 1024-point grid at $\tau =1.25\times 10^{-2}$ are visually indistinguishable from the reference at all reported times. Figure~\ref{fig:vsevnn-all-1024-tau00125} summarizes the computed solution and the absolute error of $\phi$ for the first-, second-, and third-order VS-EVNN schemes. In this run, the absolute error decreases as the temporal order increases.

\begin{figure}[H]
\centering
\begin{subfigure}[t]{0.48\textwidth}
\centering
\includegraphics[width=\linewidth]{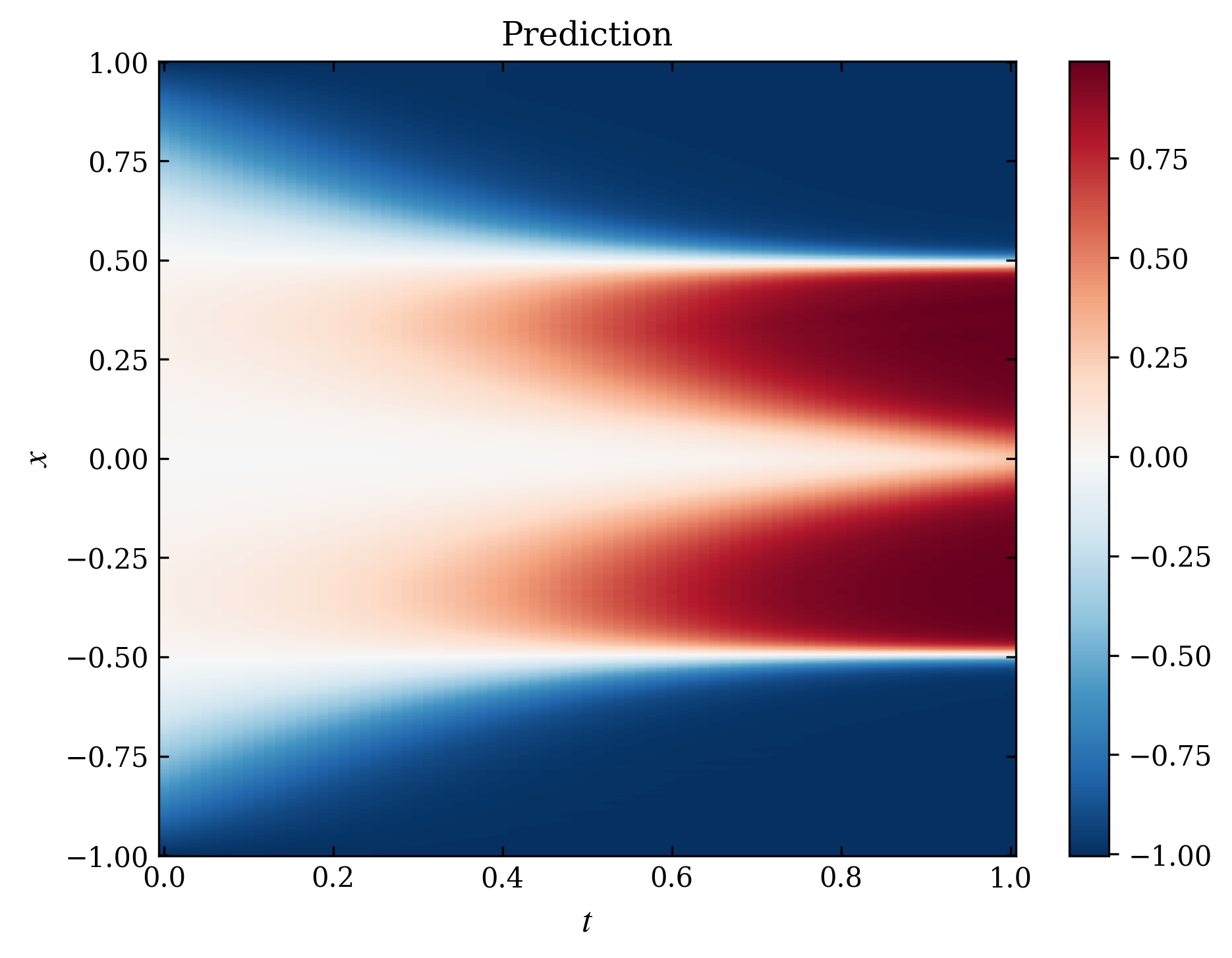}
\caption{Computed solution}
\end{subfigure}
\hfill
\begin{subfigure}[t]{0.48\textwidth}
\centering
\includegraphics[width=\linewidth]{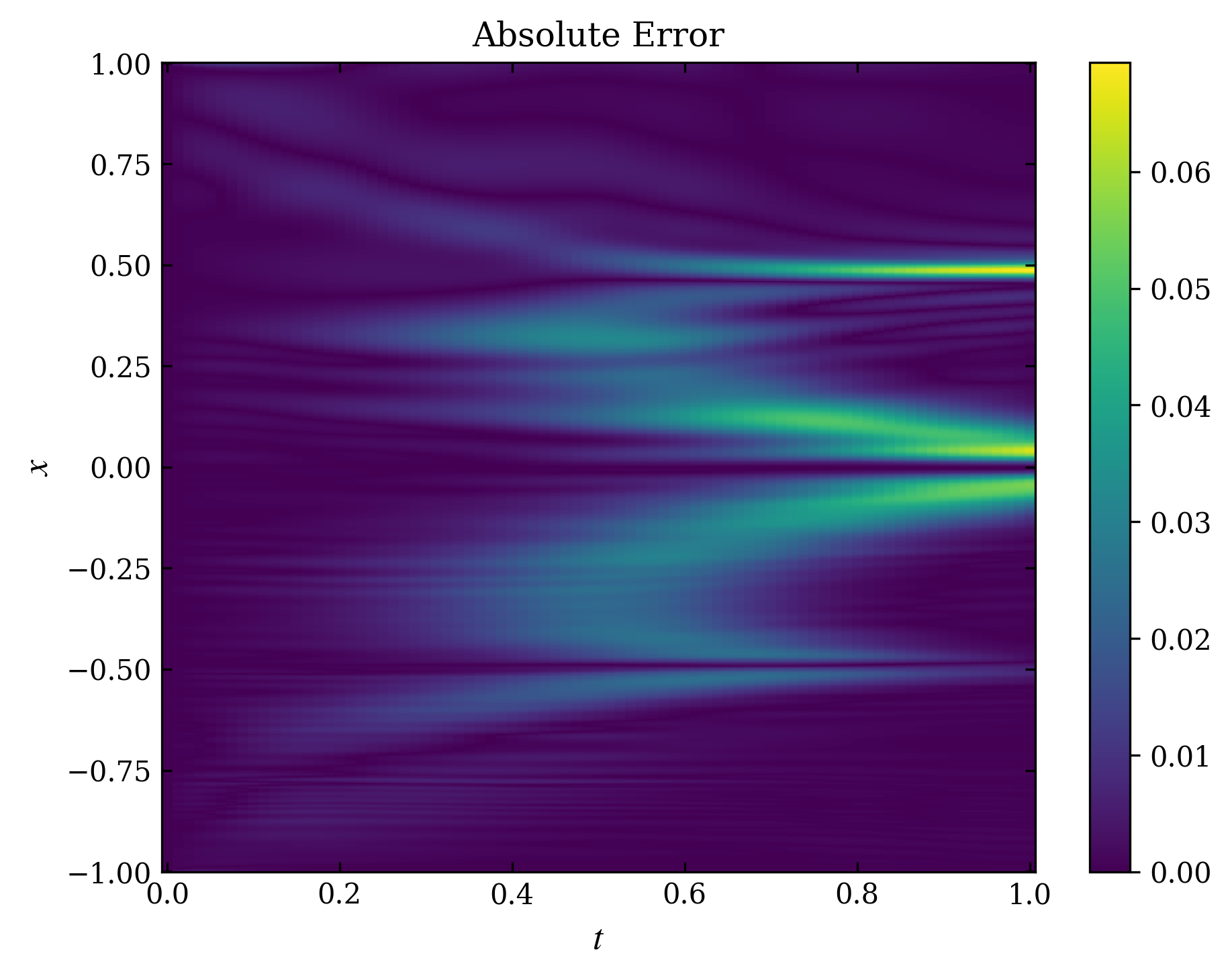}
\caption{First-order error}
\end{subfigure}
\par\bigskip
\begin{subfigure}[t]{0.48\textwidth}
\centering
\includegraphics[width=\linewidth]{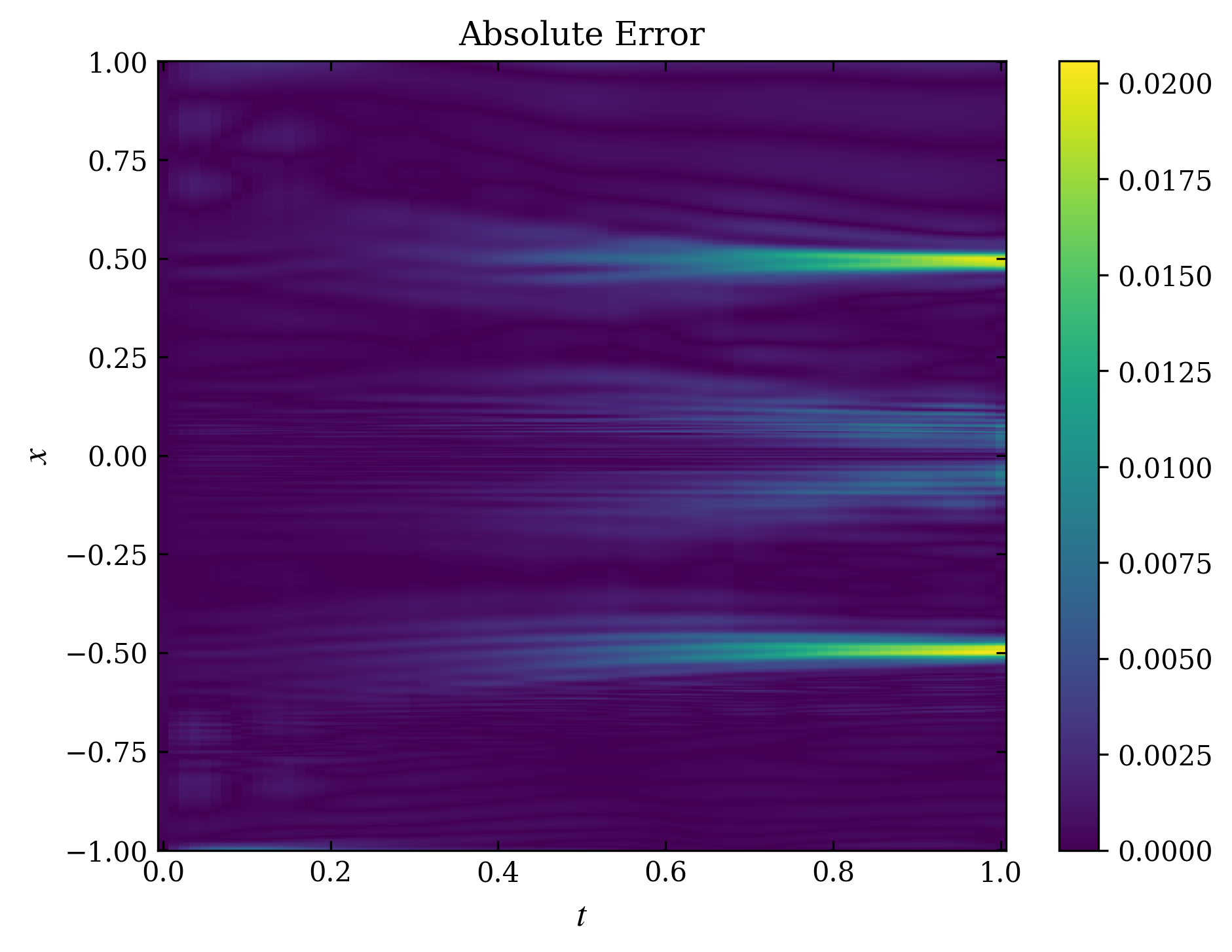}
\caption{Second-order error}
\end{subfigure}
\hfill
\begin{subfigure}[t]{0.48\textwidth}
\centering
\includegraphics[width=\linewidth]{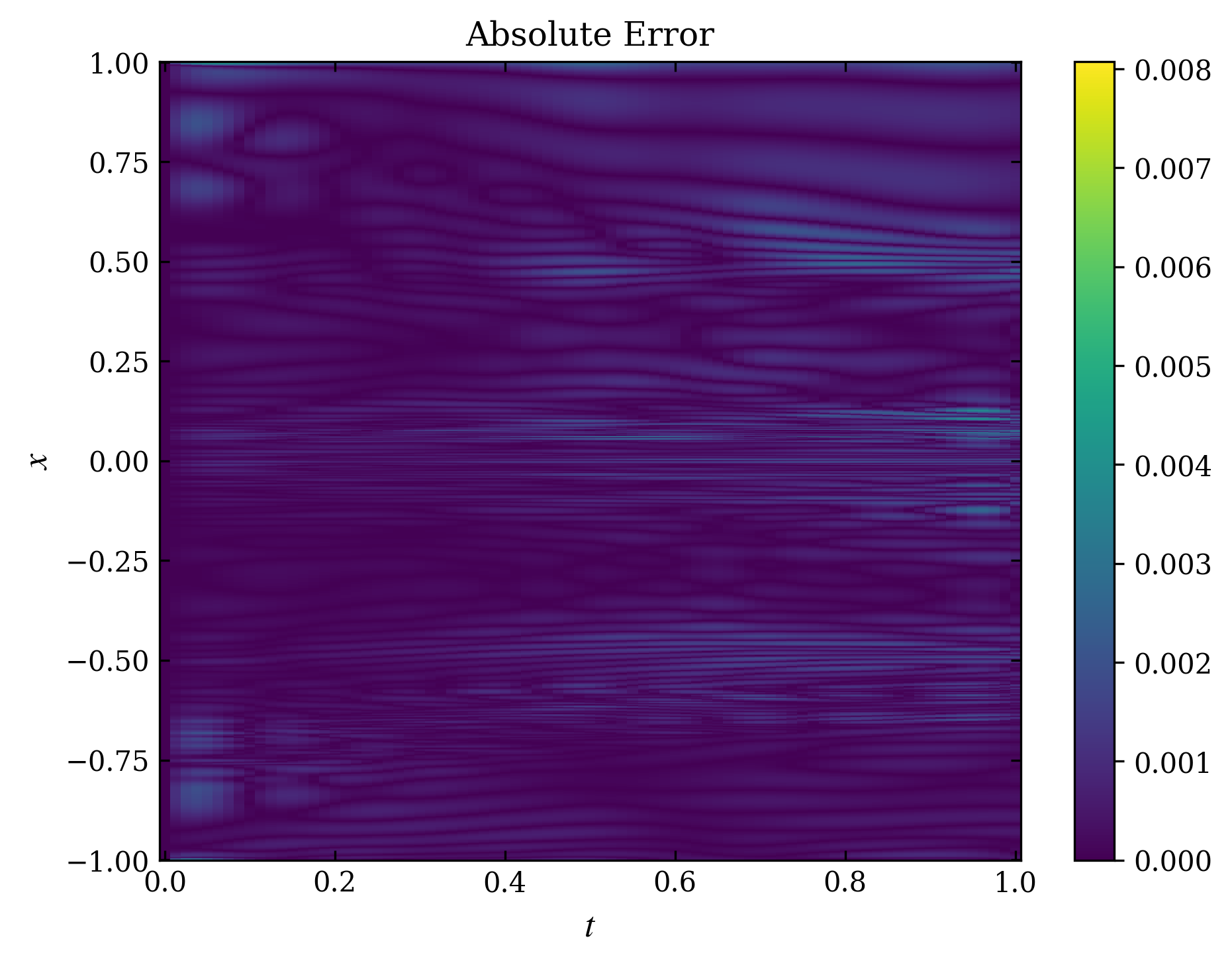}
\caption{Third-order error}
\end{subfigure}
\caption{One-dimensional Allen--Cahn equation (grid 1024, $\tau=1.25\times10^{-2}$). Computed solution and absolute errors for the first-, second-, and third-order VS-EVNN schemes. The top-left panel shows the first-order computed solution. The remaining panels show the absolute error for each scheme, each with its own color scale.}
\label{fig:vsevnn-all-1024-tau00125}
\end{figure}

Figure~\ref{fig:convergence} reports the temporal convergence of the relative $L^2$ error at $t=1$. The first-order scheme exhibits approximately first-order behavior over the tested range of time steps. At any fixed $\tau$, the higher-order schemes are much more accurate. For example, at $\tau=1.25 \times 10^{-2}$, the error drops from $1.71\times10^{-2}$ (first order) to $1.20\times10^{-3}$ (second order) and $7.10\times10^{-4}$ (third order). The third-order curve flattens at smaller time steps, consistent with non-temporal contributions becoming significant. The second-order curve continues to decrease over much of the displayed range.

\begin{figure}[H]
\centering
\includegraphics[width=0.5\textwidth]{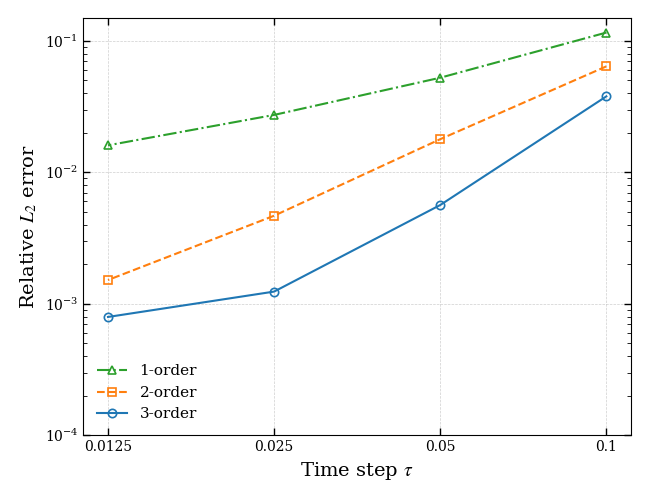}
\caption{One-dimensional Allen--Cahn equation (spatial grid 1024). Temporal convergence of the relative $L^2$ error at $t=1$ for the first-, second-, and third-order VS-EVNN schemes at various time step sizes $\tau$.}
\label{fig:convergence}
\end{figure}

\subsubsection{One-dimensional Cahn--Hilliard equation} In the coarsening regime, we simulate to $T_{\text{final}}=1000$. The dynamics separate the smooth initial profile into a two-phase square-wave pattern with plateaus at $\phi=\pm1$ and interfaces near $x=\pm0.5$, which then relaxes slowly toward the steady state. The small mobility slows the evolution in the chosen time units, and we test large time steps over this long physical-time interval. Figure~\ref{fig:1d-ch-solution} shows the computed solutions of the first-, second-, and third-order VS-EVNN schemes at the step $\tau=100$, which corresponds to only 10 time steps over the entire simulation. It also reports the annotated relative $L^2$ errors at $t=200$, $600$, and $1000$. All three orders capture the profile. The errors decrease in time as the dynamics slow down. 

\begin{figure}[H]
\centering
\begin{subfigure}[t]{0.85\textwidth}
\centering
\includegraphics[width=\linewidth]{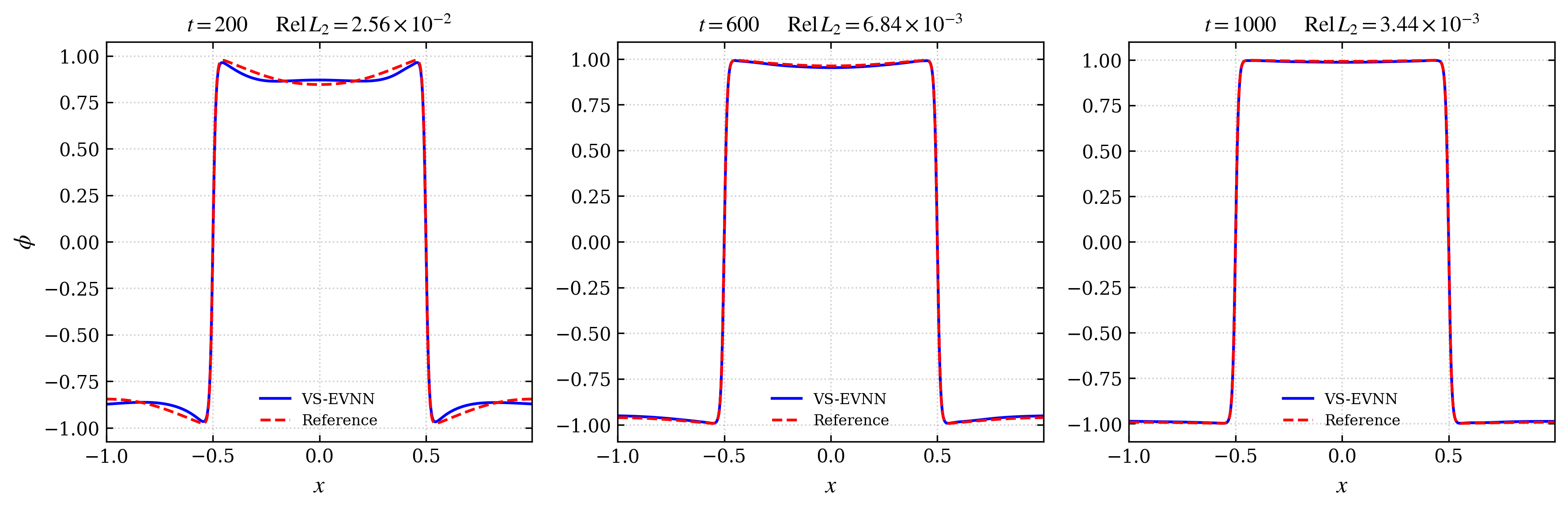}
\caption{First-order VS-EVNN, $\tau=100$}
\end{subfigure}

\begin{subfigure}[t]{0.85\textwidth}
\centering
\includegraphics[width=\linewidth]{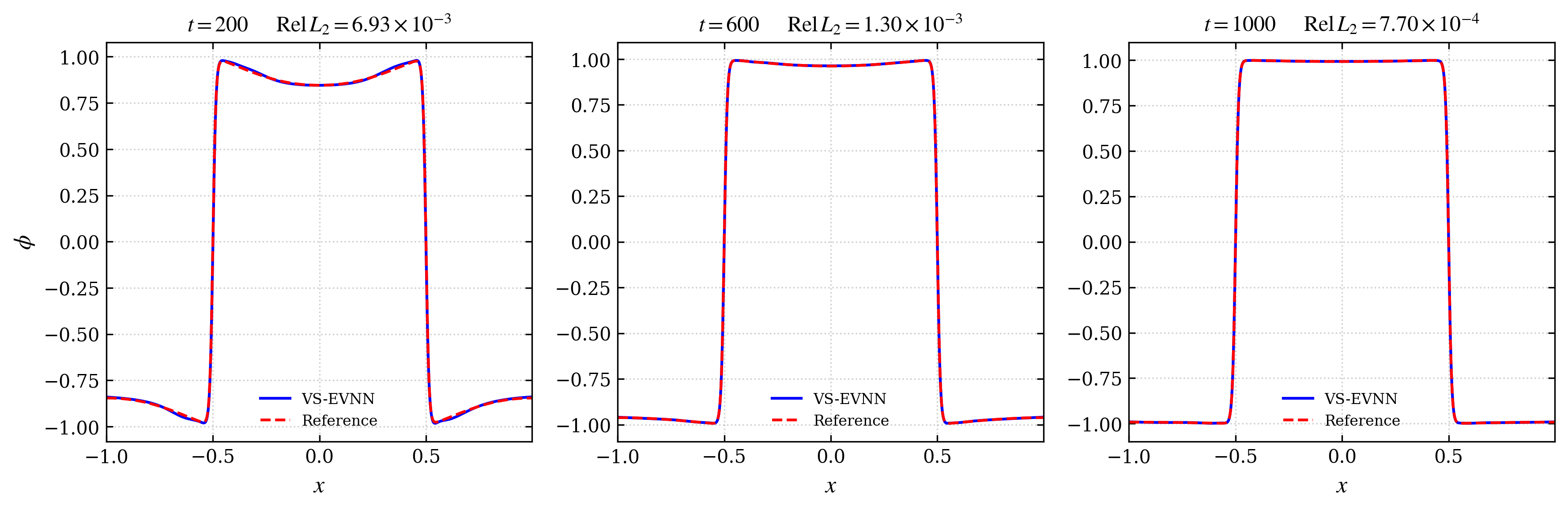}
\caption{Second-order VS-EVNN, $\tau=100$}
\end{subfigure}

\begin{subfigure}[t]{0.85\textwidth}
\centering
\includegraphics[width=\linewidth]{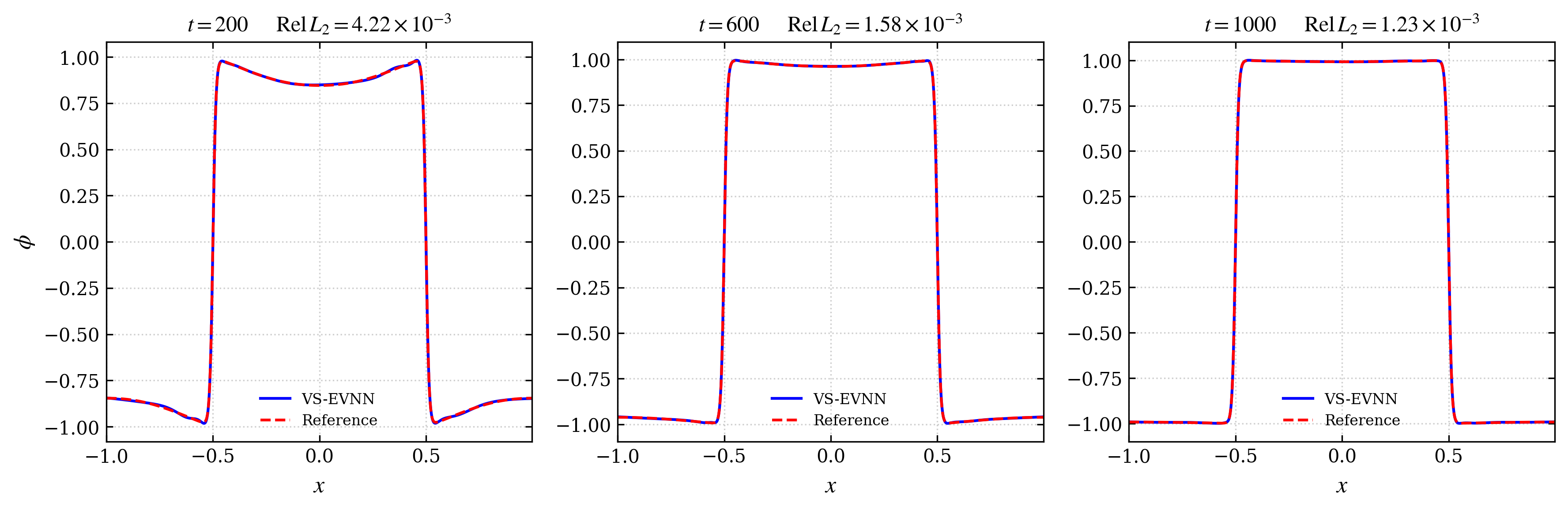}
\caption{Third-order VS-EVNN, $\tau=100$}
\end{subfigure}
\caption{One-dimensional Cahn--Hilliard equation ($M=2\times10^{-4}$, $\varepsilon=0.01$). Computed solutions (solid) and reference (dashed) at $t=200$, $600$, and $1000$ for the first-, second-, and third-order VS-EVNN schemes at $\tau=100$, with the relative $L^2$ errors annotated above each panel.}
\label{fig:1d-ch-solution}
\end{figure}

Figure~\ref{fig:1d-ch-energy} reports the structural diagnostics. The computed energy decays monotonically for all three orders and matches the reference decay. The first-order curve visibly lags the reference during the fast transient ($t\lesssim600$), while the second- and third-order curves overlap it, and the relative energy errors of the higher-order schemes remain well below those of the first-order curve. The mass error stays within $\pm10^{-7}$ throughout, which is the single-precision round-off level, consistent with the exact mean-value projection of Section~\ref{sec:nn_approximation}.

\begin{figure}[H]
\centering
\begin{subfigure}[t]{0.32\textwidth}
\centering
\includegraphics[width=\linewidth]{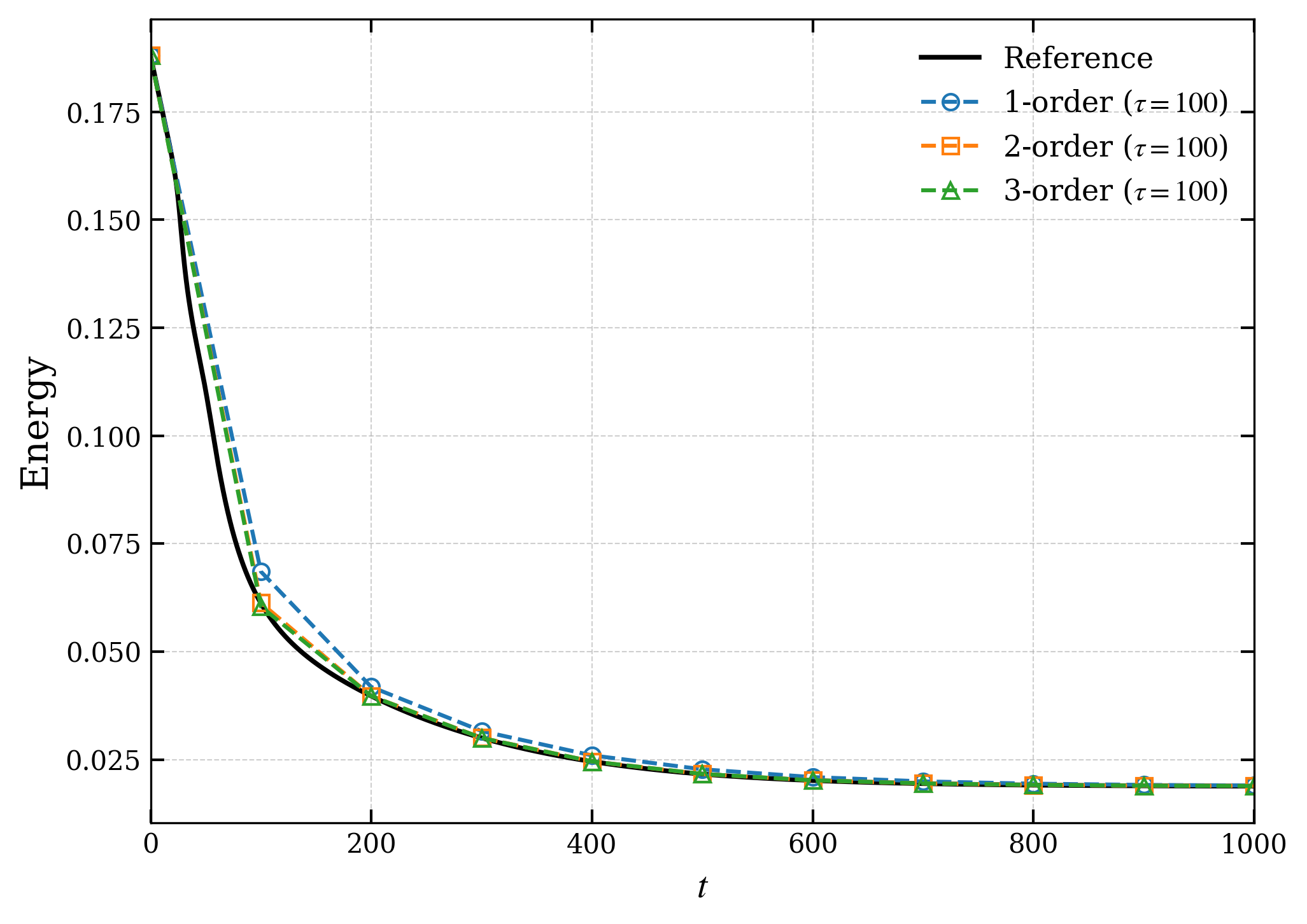}
\caption{Energy evolution}
\end{subfigure}\hfill
\begin{subfigure}[t]{0.32\textwidth}
\centering
\includegraphics[width=\linewidth]{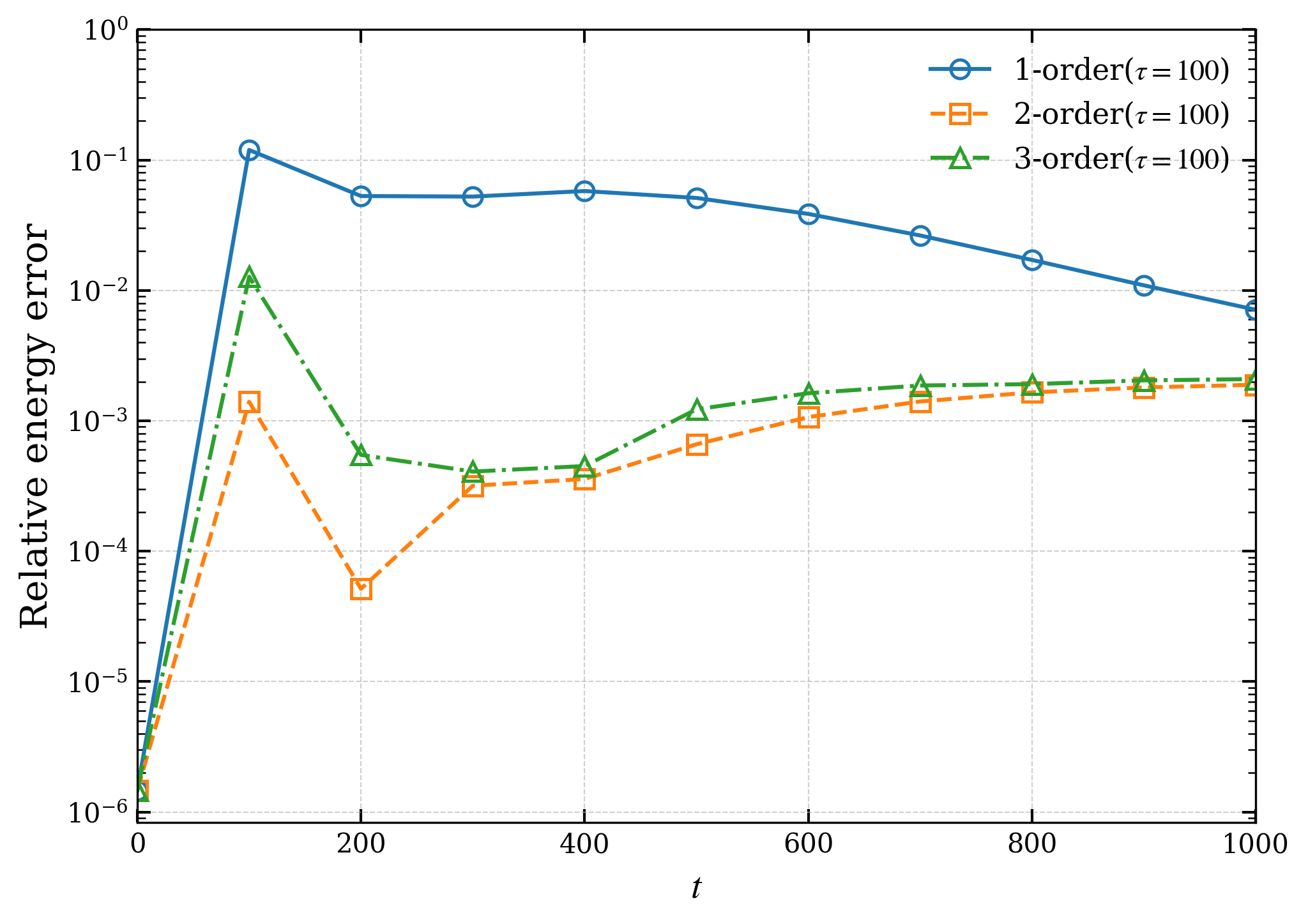}
\caption{Relative energy error}
\end{subfigure}\hfill
\begin{subfigure}[t]{0.32\textwidth}
\centering
\includegraphics[width=\linewidth]{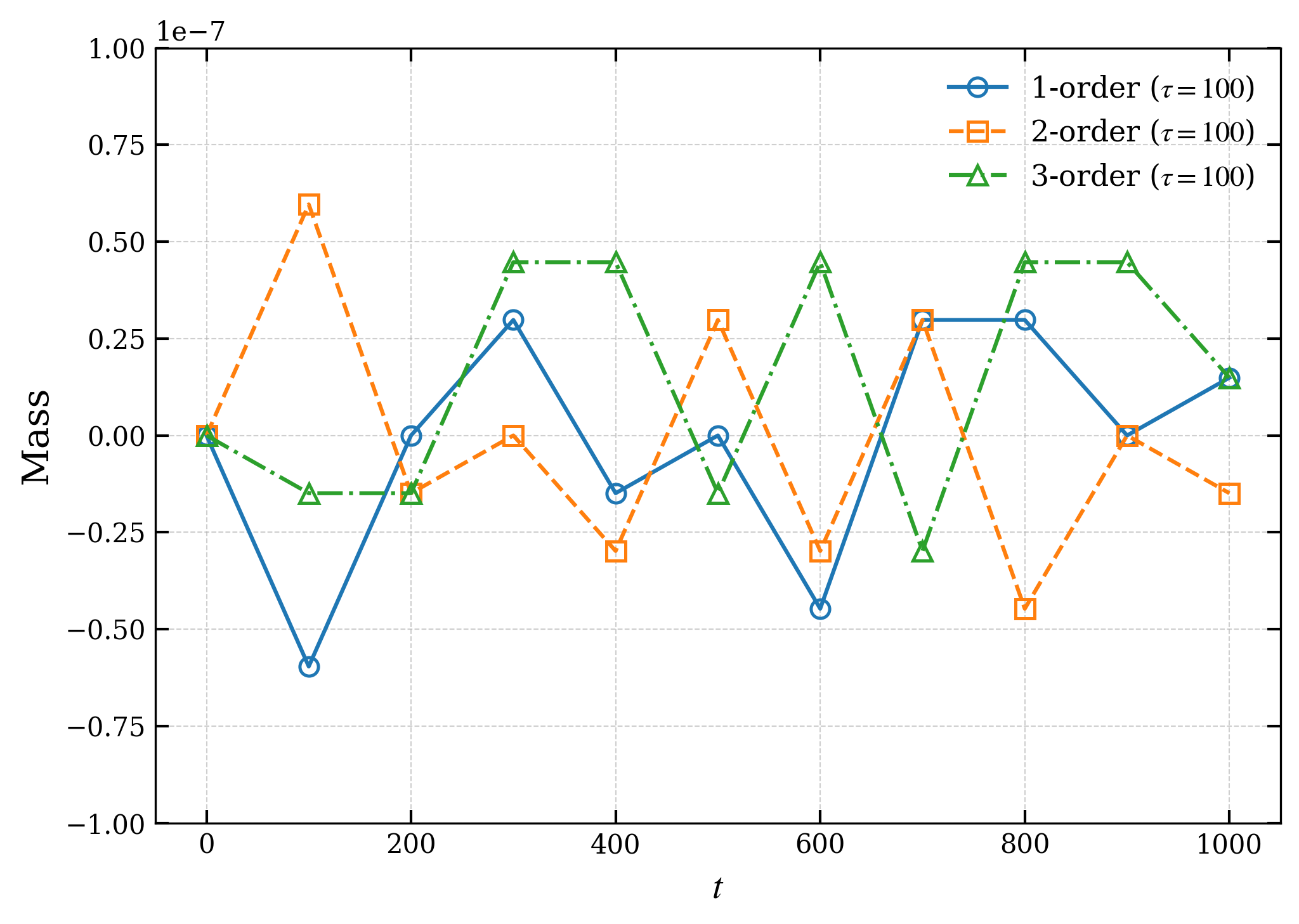}
\caption{Mass error}
\end{subfigure}
\caption{One-dimensional Cahn--Hilliard equation ($M=2\times10^{-4}$, $\varepsilon=0.01$, $\tau=100$). (a) Energy evolution of the first-, second-, and third-order VS-EVNN schemes against the reference. (b) Relative energy error over time. (c) Deviation of the total mass from its initial value, which remains at the single-precision round-off level.}
\label{fig:1d-ch-energy}
\end{figure}

As a more demanding test, we take $M=1$ and $\varepsilon=0.002$ in the sharp-interface regime on the finer grid $N=4096$. Figure~\ref{fig:1d-ch-sharp} shows the third-order VS-EVNN solution at the step $\tau=0.1$. The computed profile is visually close to the reference from $t=0.2$ onward, and the relative $L^2$ error decreases to $1.32\times10^{-3}$ by $t=0.3$. Figure~\ref{fig:1d-ch-sharp-energy} shows these energy histories at the per-order step sizes stated in the caption. The mass error again stays at the round-off level, as in Figure~\ref{fig:1d-ch-energy}.

\begin{figure}[H]
\centering
\includegraphics[width=0.85\textwidth]{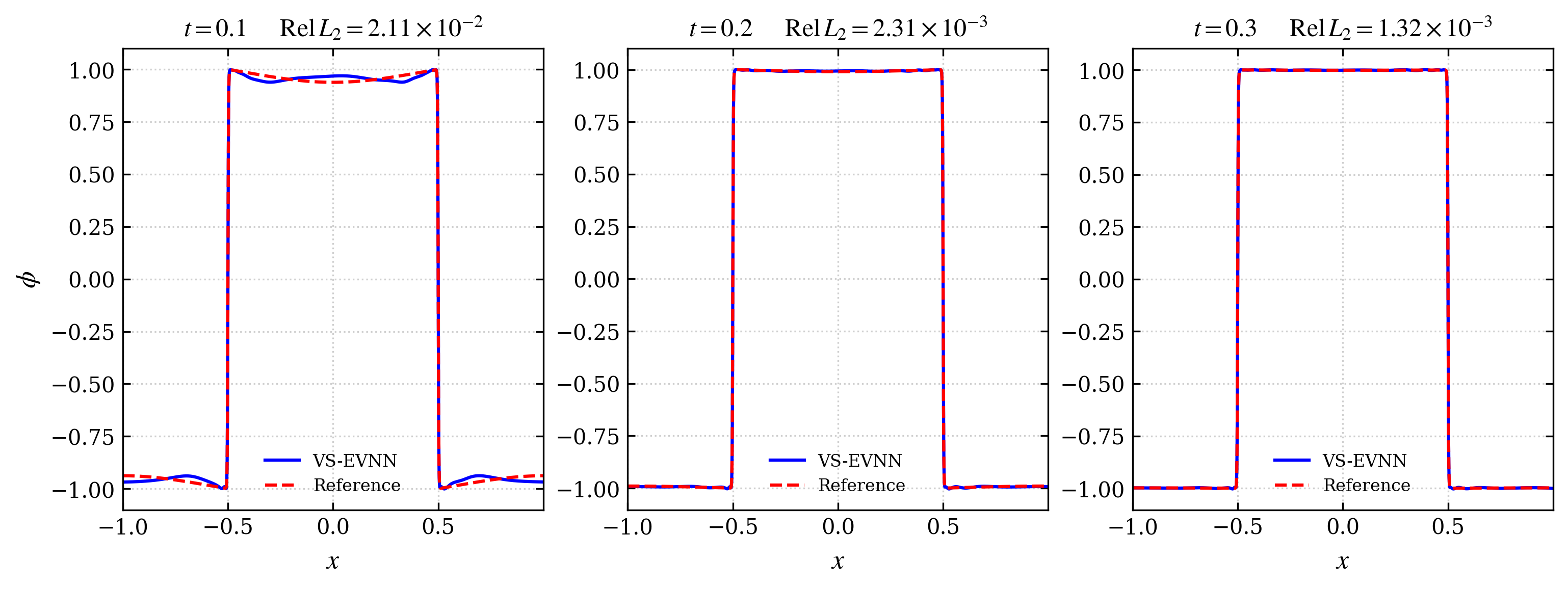}
\caption{One-dimensional Cahn--Hilliard equation with a sharp interface ($M=1$, $\varepsilon=0.002$, $N=4096$). Third-order VS-EVNN solution (solid) at $\tau=0.1$ and reference (dashed) at $t=0.1$, $0.2$, and $0.3$, with the relative $L^2$ errors annotated above each panel.}
\label{fig:1d-ch-sharp}
\end{figure}

\begin{figure}[H]
\centering
\includegraphics[width=0.48\textwidth]{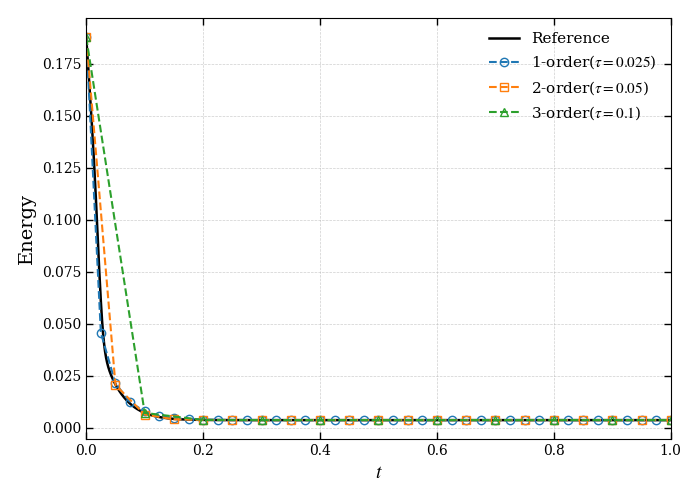}
\caption{One-dimensional Cahn--Hilliard equation with a sharp interface ($M=1$, $\varepsilon=0.002$). Energy evolution of the VS-EVNN schemes against the reference. The first-, second-, and third-order schemes use $\tau=0.025$, $0.05$, and $0.1$, respectively, so this is not a comparison at a common time step.}
\label{fig:1d-ch-sharp-energy}
\end{figure}

\subsubsection{Two-dimensional Allen--Cahn equation} The two-dimensional Allen--Cahn runs use the two initial conditions of Section~\ref{subsec:test-problems} with the time step $\tau = 1$. Fixed-step results are reported up to $T=20$.

\paragraph{Star-shaped interface.} Figure~\ref{fig:2d-star} compares the solutions obtained by the VS-EVNN schemes of different orders, and Figure~\ref{fig:2d-ac-errors}(a,c) reports the corresponding relative $L^2$ and energy errors over time.

\begin{figure}[H]
\centering
\includegraphics[width=0.7\textwidth]{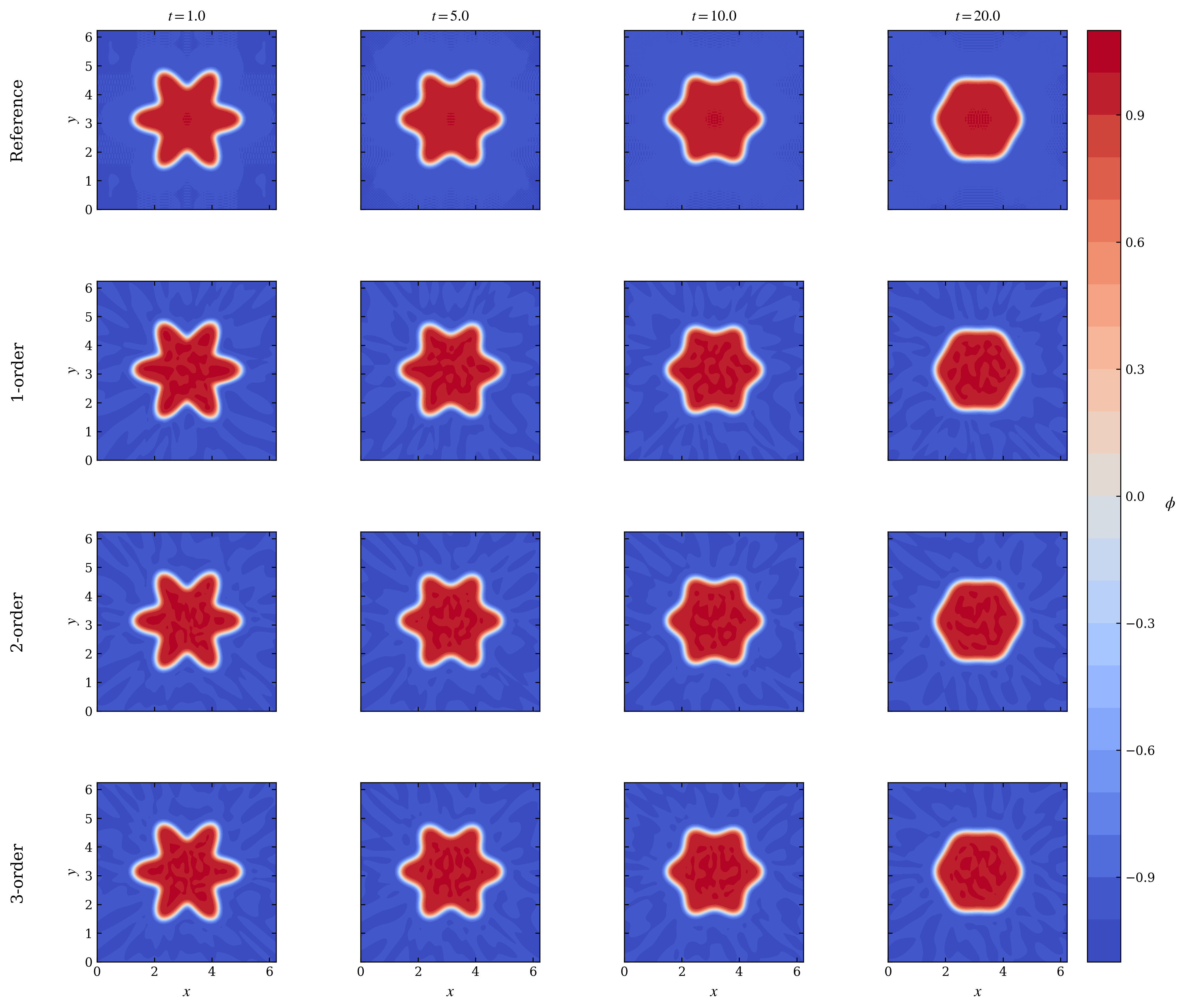}
\caption{Two-dimensional Allen--Cahn equation, star-shaped interface ($\tau=1$). Solution snapshots at $t=1, 5, 10, 20$ for the first-, second-, and third-order VS-EVNN schemes, with the reference in the top row.}
\label{fig:2d-star}
\end{figure}

\begin{figure}[H]
\centering
\begin{subfigure}[t]{0.48\textwidth}
\includegraphics[width=\textwidth]{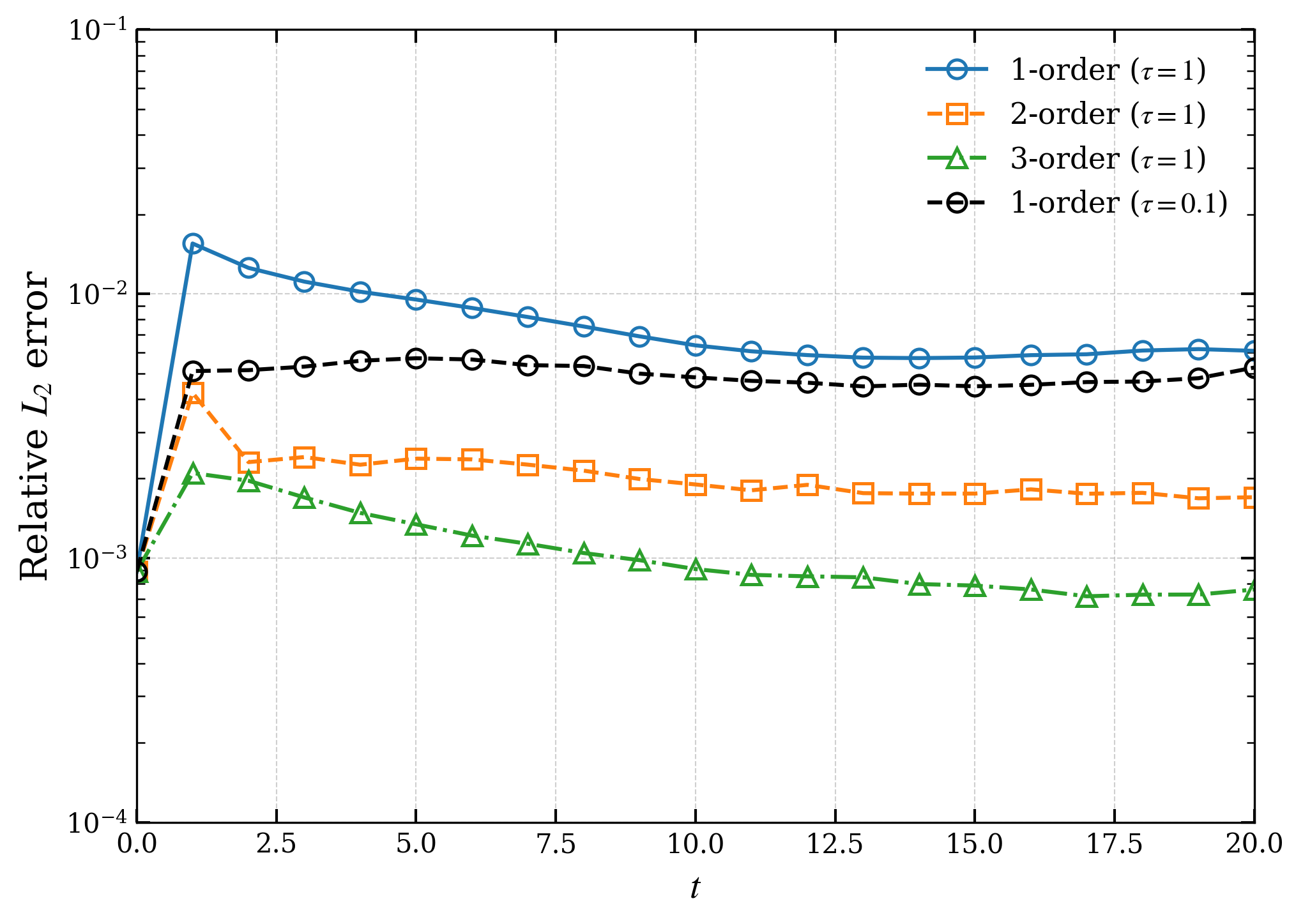}
\caption{Star: relative $L^2$ error}
\label{fig:2d-ac-errors-a}
\end{subfigure}\hfill
\begin{subfigure}[t]{0.48\textwidth}
\includegraphics[width=\textwidth]{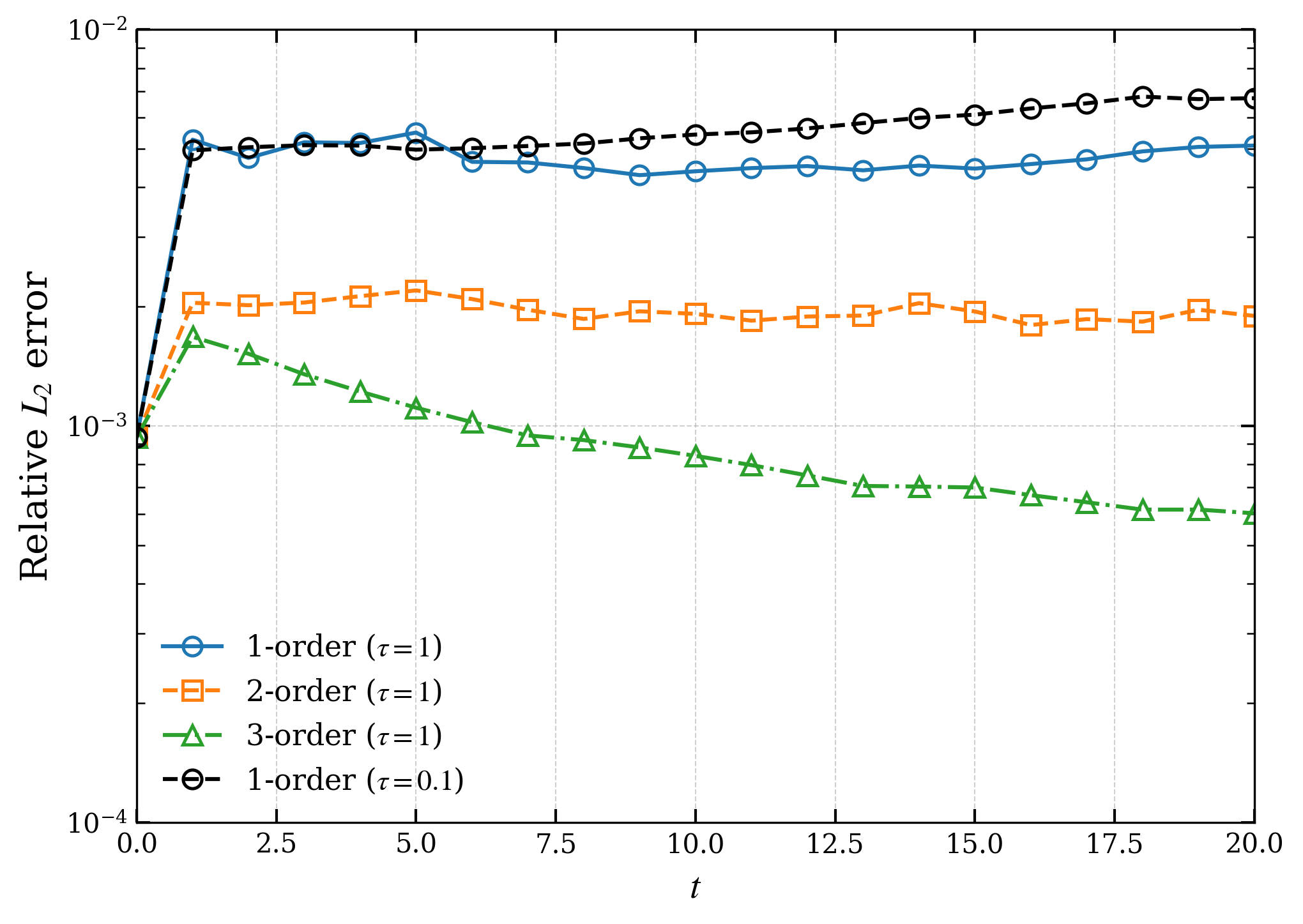}
\caption{Circles: relative $L^2$ error}
\label{fig:2d-ac-errors-b}
\end{subfigure}\hfill

\begin{subfigure}[t]{0.48\textwidth}
\includegraphics[width=\textwidth]{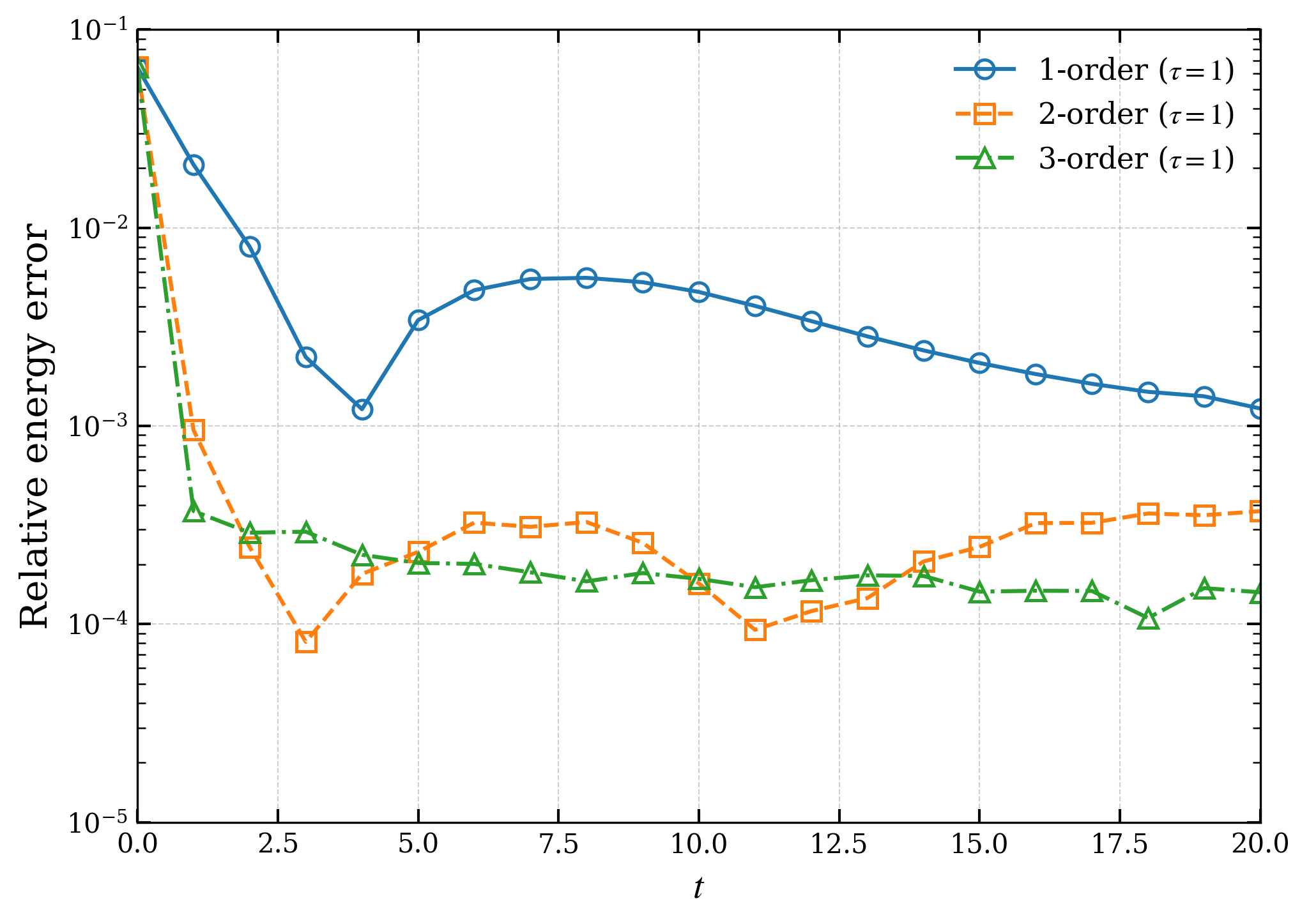}
\caption{Star: relative energy error}
\label{fig:2d-ac-errors-c}
\end{subfigure}\hfill
\begin{subfigure}[t]{0.48\textwidth}
\includegraphics[width=\textwidth]{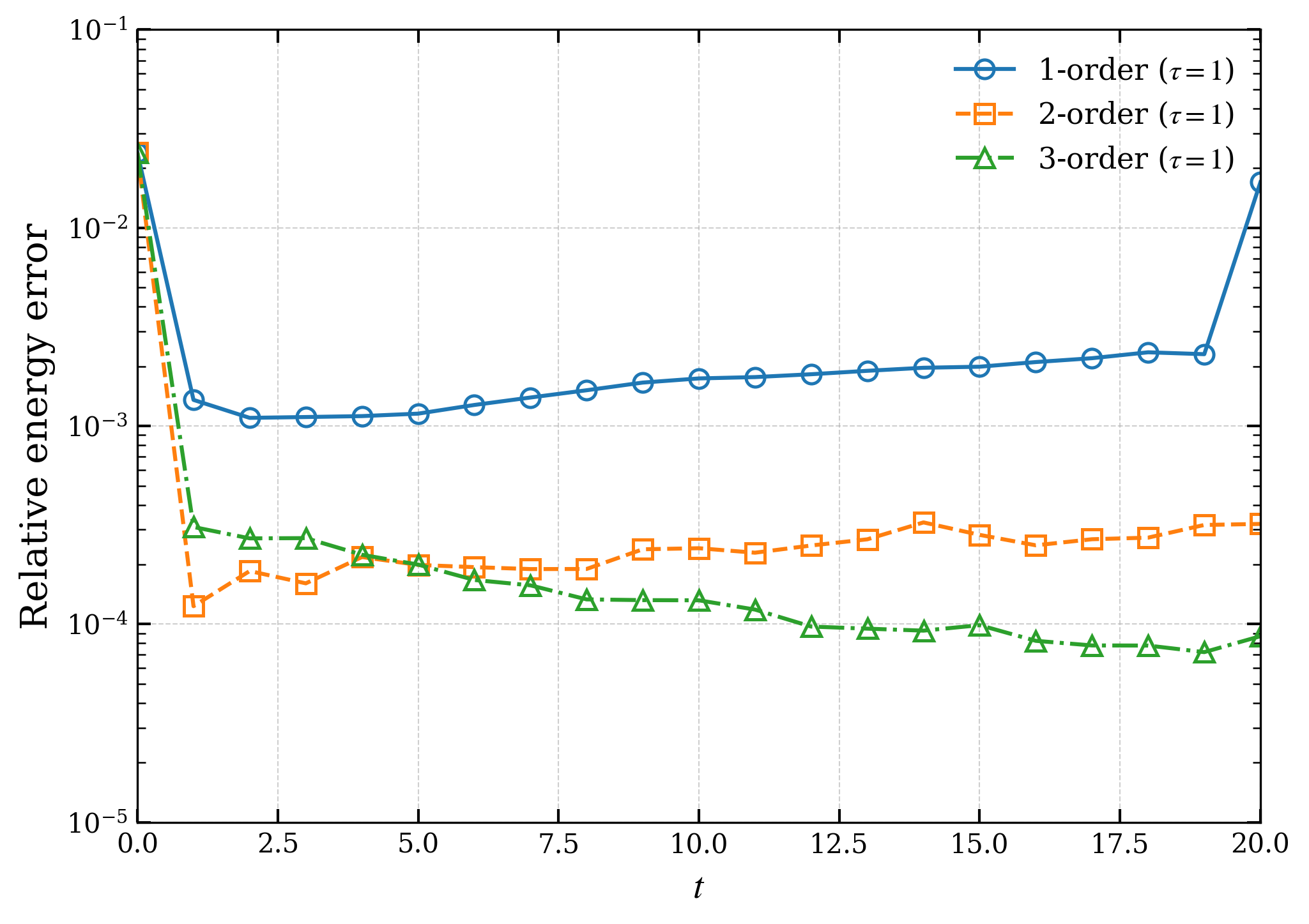}
\caption{Circles: relative energy error}
\label{fig:2d-ac-errors-d}
\end{subfigure}\hfill
\caption{Two-dimensional Allen--Cahn equation ($\tau=1$). Relative $L^2$ errors (top row) and relative energy errors (bottom row) of the first-, second-, and third-order VS-EVNN schemes for the star-shaped interface (a, c) and the randomly placed circles (b, d). Each top-row panel also includes a refined first-order run at $\tau=0.1$ to probe sensitivity to the time step. The bottom-row panels show the three orders only at $\tau=1$.}
\label{fig:2d-ac-errors}
\end{figure}

All three schemes capture the curvature-driven evolution in Figure~\ref{fig:2d-star}. The high-curvature tips of the star retract, and the interface relaxes toward a convex, nearly circular shape by $t=20$. At the snapshot scale, the computed morphologies are close to the reference. The error histories quantify the differences between orders. After the initial transient, the higher-order relative $L^2$ errors generally decrease as the interface becomes smoother, although the error histories are not monotone (Figure~\ref{fig:2d-ac-errors-a}). The high-order VS-EVNN schemes at $\tau=1$ outperform the first-order scheme at the smaller time step $\tau=0.1$.  At $t=1$, the relative $L^2$ error falls from $1.55\times10^{-2}$ for the first-order scheme to $2.10\times10^{-3}$ for the third-order scheme, and the relative energy errors of the higher-order schemes are smaller than those of the first-order scheme over the displayed evolution (Figure~\ref{fig:2d-ac-errors-c}).

\paragraph{Randomly placed circles.}
\label{subsec:2d-random-circles}
Figure~\ref{fig:2d-random} compares the VS-EVNN solutions across orders for the randomly placed circles. The corresponding error histories are shown in Figure~\ref{fig:2d-ac-errors}(b,d). To assess robustness with respect to the random draw, we report in Table~\ref{tab:2d-seeds} the relative $L^2$ errors, averaged over $t \in [0,20]$, for five different random seeds. The headline seed-42 run of the preceding figures is reported separately and is not one of the five.

\begin{figure}[H]
\centering
\includegraphics[width=0.65\textwidth]{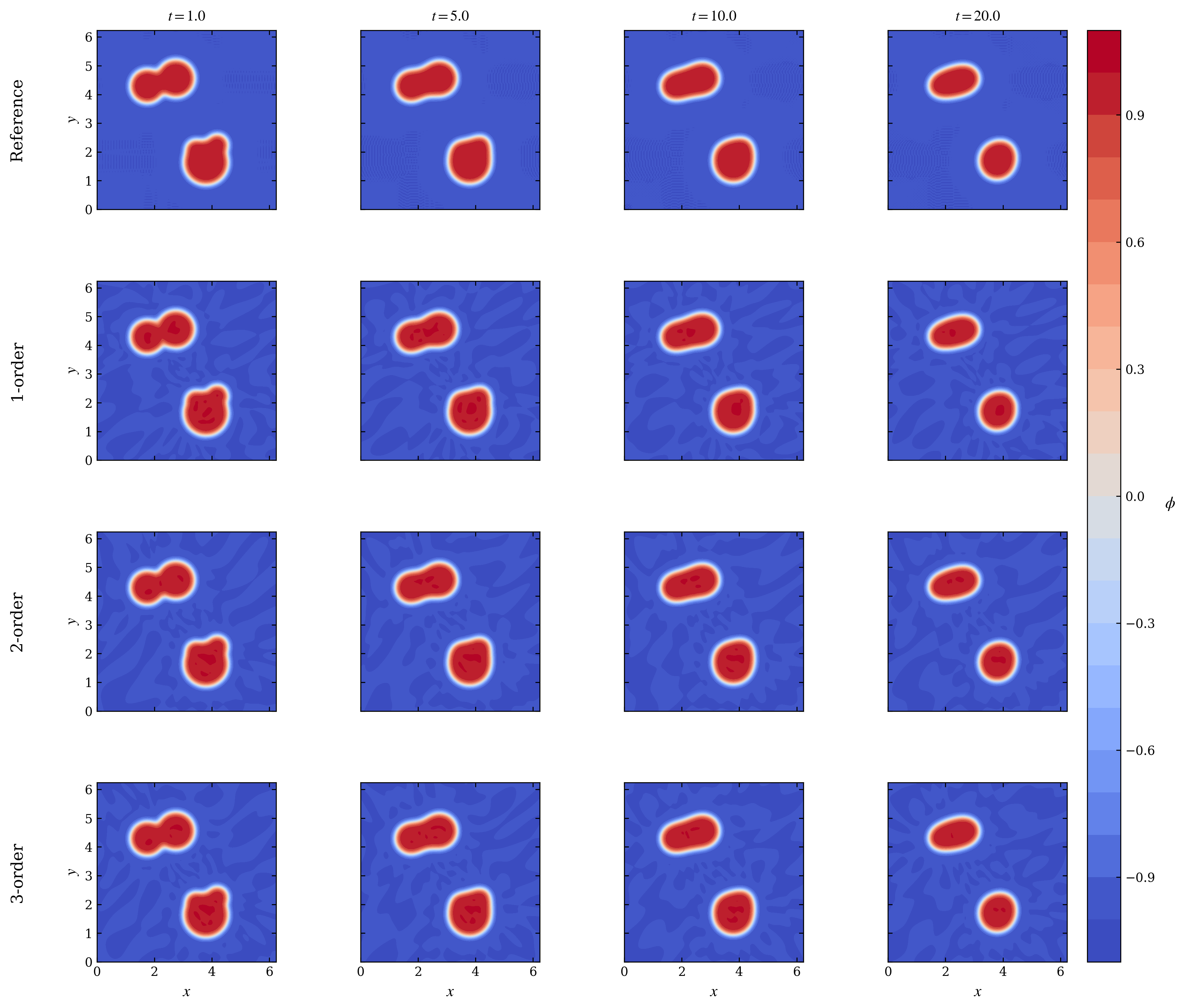}
\caption{Two-dimensional Allen--Cahn equation, randomly placed circles ($\tau=1$). Solution snapshots at $t=1, 5, 10, 20$ for the first-, second-, and third-order VS-EVNN schemes, with the reference in the top row.}
\label{fig:2d-random}
\end{figure}

\begin{table}[H]
\centering
\caption{Two-dimensional Allen--Cahn equation, randomly placed circles. Relative $L^2$ errors under different random seeds (averaged over $t \in [0,20]$, $\tau=1$).}
\label{tab:2d-seeds}
\begin{tabular}{c c c c}
\toprule
Seed & 1st order & 2nd order & 3rd order \\
\midrule
1 & $2.55\times10^{-2}$ & $1.90\times10^{-3}$ & $1.13\times10^{-3}$ \\
2 & $8.89\times10^{-3}$ & $1.41\times10^{-3}$ & $1.16\times10^{-3}$ \\
3 & $7.80\times10^{-3}$ & $1.48\times10^{-3}$ & $1.29\times10^{-3}$ \\
4 & $1.57\times10^{-2}$ & $1.68\times10^{-3}$ & $1.53\times10^{-3}$ \\
5 & $5.39\times10^{-3}$ & $1.87\times10^{-3}$ & $8.71\times10^{-4}$ \\
\midrule
Mean & $1.27\times10^{-2}$ & $1.67\times10^{-3}$ & $1.20\times10^{-3}$ \\
\bottomrule
\end{tabular}
\end{table}

Two effects, both visible in Figure~\ref{fig:2d-random}, drive the coarsening dynamics. An overlapping pair of circles merges into a single convex domain, and the domains slowly shrink under curvature flow. All three schemes reproduce this evolution. The first-order error grows slowly in time, from $4.96\times10^{-3}$ at $t=1$ to $6.72\times10^{-3}$ at $t=20$, while the third-order error decreases to $6.02\times10^{-4}$. After the initial transient, the energy errors of the higher-order schemes are on the order of $10^{-4}$ (Figure~\ref{fig:2d-ac-errors-d}). The seed study in Table~\ref{tab:2d-seeds} indicates that this advantage is robust. Across five random initial conditions, the first-order error varies by a factor of about five, whereas the second- and third-order errors stay within narrow bands around $1.7\times10^{-3}$ and $1.2\times10^{-3}$, respectively.

\subsubsection{Two-dimensional Cahn--Hilliard equation} The two-dimensional Cahn--Hilliard runs use the randomly placed circles initial condition of Section~\ref{subsec:test-problems} with the setup of Section~\ref{subsec:implementation}. The final time is $T_{\text{final}}=100$, and the fixed-step runs use $\tau=20$ for all three orders.

\begin{figure}[H]
\centering
\includegraphics[width=0.65\textwidth]{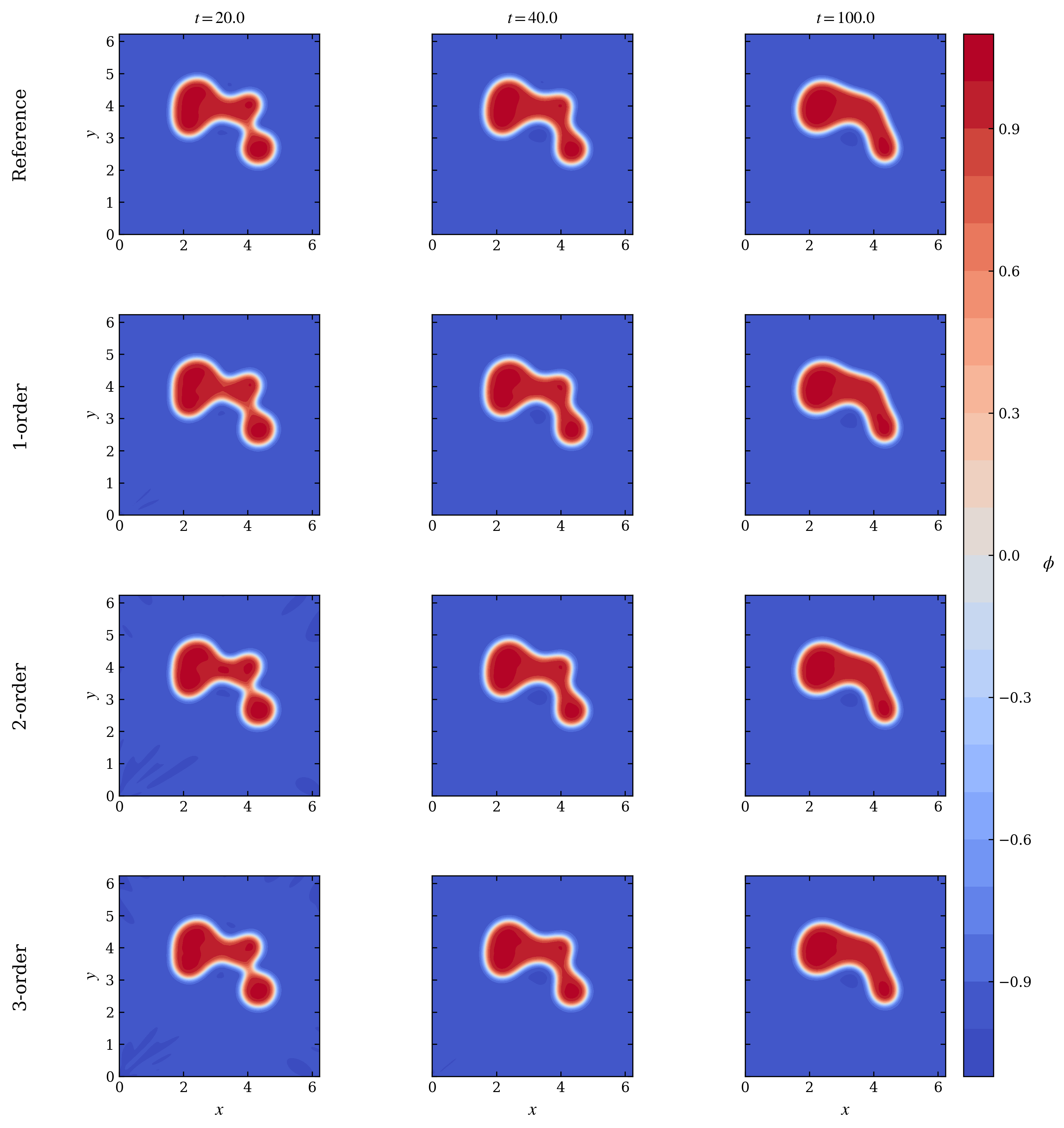}
\caption{Two-dimensional Cahn--Hilliard equation, randomly placed circles ($\tau=20$). Solution snapshots at $t=20$, $40$, and $100$ for the first-, second-, and third-order VS-EVNN schemes, with the reference in the top row.}
\label{fig:2d-ch-random}
\end{figure}

\begin{figure}[H]
\centering
\begin{subfigure}[t]{0.32\textwidth}
\includegraphics[width=\textwidth]{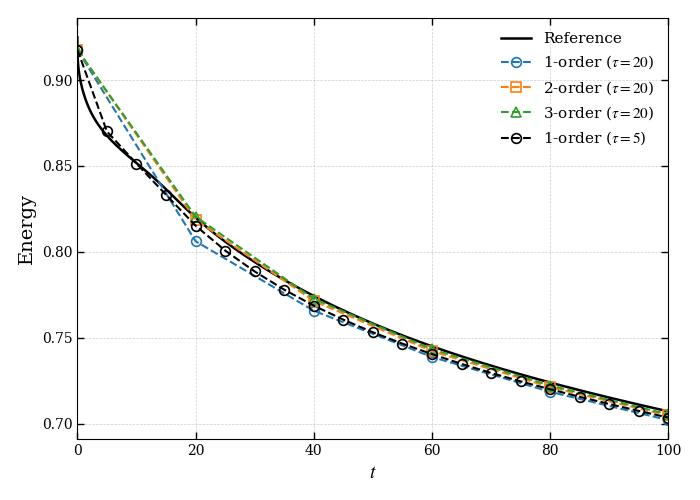}
\caption{Energy evolution}
\end{subfigure}\hfill
\begin{subfigure}[t]{0.32\textwidth}
\includegraphics[width=\textwidth]{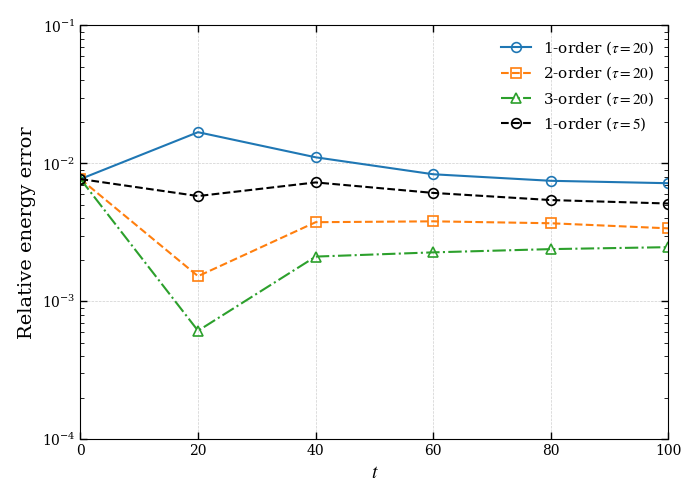}
\caption{Relative energy error}
\end{subfigure}\hfill
\begin{subfigure}[t]{0.32\textwidth}
\includegraphics[width=\textwidth]{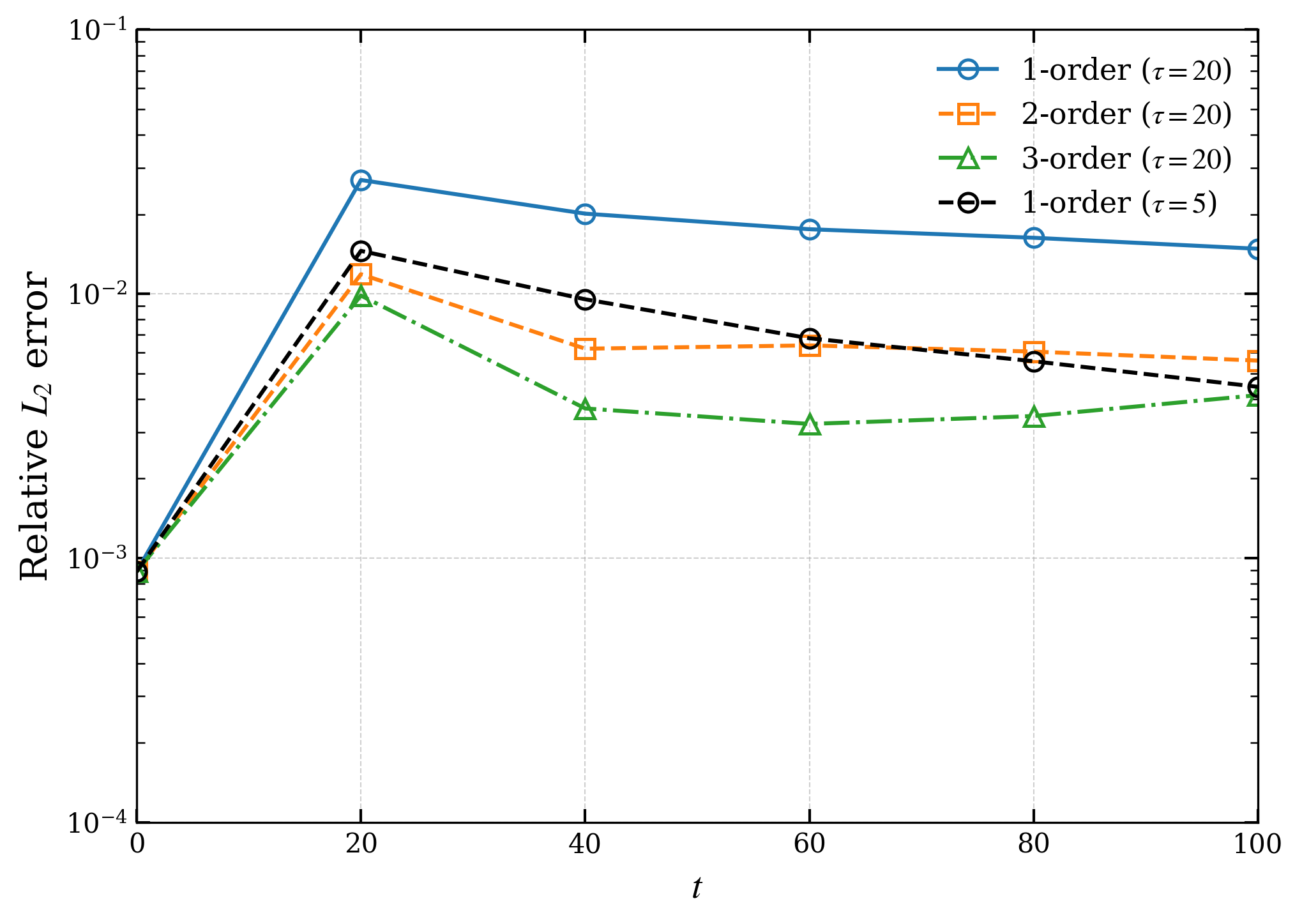}
\caption{Relative $L^2$ error}
\end{subfigure}
\caption{Two-dimensional Cahn--Hilliard equation, randomly placed circles ($\tau=20$). (a) Energy evolution of the first-, second-, and third-order VS-EVNN schemes against the reference. (b) Relative energy error over time. (c) Relative $L^2$ error of the solution. All three panels also include a first-order run at the smaller step $\tau=5$ for comparison.}
\label{fig:2d-ch-random-diag}
\end{figure}

Unlike the Allen--Cahn evolution of Section~\ref{subsec:2d-random-circles}, where isolated domains shrink and vanish under curvature flow, the conserved dynamics coarsen by mass redistribution. In Figure~\ref{fig:2d-ch-random}, the overlapping circles merge into a single domain whose area is determined by the conserved mass, and the small satellite domain is absorbed via $H^{-1}$ transport rather than disappearing pointwise. All three schemes reproduce this evolution at the large step $\tau=20$. The second- and third-order snapshots closely match the reference at the displayed resolution, while the first-order scheme slightly smooths the neck region during the merging phase around $t=40$. The energy decays monotonically for all orders, and the relative energy errors and relative $L^2$ errors of the higher-order schemes are generally smaller than those of the first-order scheme at the same time step (Figure~\ref{fig:2d-ch-random-diag}).

\subsection{Adaptive time stepping and long-time integration}
\label{subsec:adaptive-results}

\subsubsection{Tolerance study for the one-dimensional Allen--Cahn equation} We apply the adaptive strategy of Section~\ref{subsec:time-adaptive} to the first-order VS-EVNN. Figure~\ref{fig:1d-timestep} plots the adaptively selected time step for tolerances $\mathrm{tol}\in\{0.05, 0.02, 0.01\}$, while Figure~\ref{fig:1d-timestep-order} compares the time-step histories of the first-, second-, and third-order VS-EVNN at the fixed $\mathrm{tol}=0.01$. The higher-order schemes admit much larger steps. Table~\ref{tab:efficiency_accuracy} reports the accepted step counts and the relative $L^2$ errors for the adaptive runs. As $\mathrm{tol}$ decreases, the number of accepted steps increases and the relative errors decrease. At $\mathrm{tol}=0.01$, the higher-order schemes need fewer accepted steps than the first-order scheme and reach smaller relative $L^2$ errors. For all tolerances, the adaptive step size is small during the initial period of rapid change in the solution and then generally increases over time, with intermediate reductions.

\begin{figure}[H]
\centering
\begin{subfigure}[t]{0.48\textwidth}
\centering
\includegraphics[width=\linewidth]{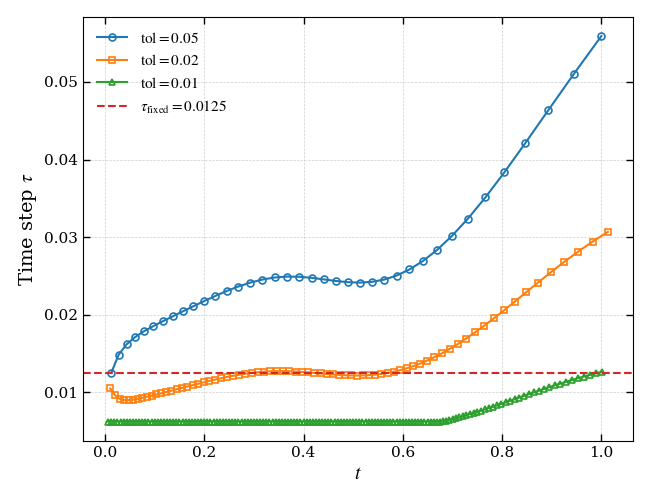}
\caption{First-order VS-EVNN for $\mathrm{tol}\in\{0.05, 0.02, 0.01\}$}
\label{fig:1d-timestep}
\end{subfigure}\hfill
\begin{subfigure}[t]{0.48\textwidth}
\centering
\includegraphics[width=\linewidth]{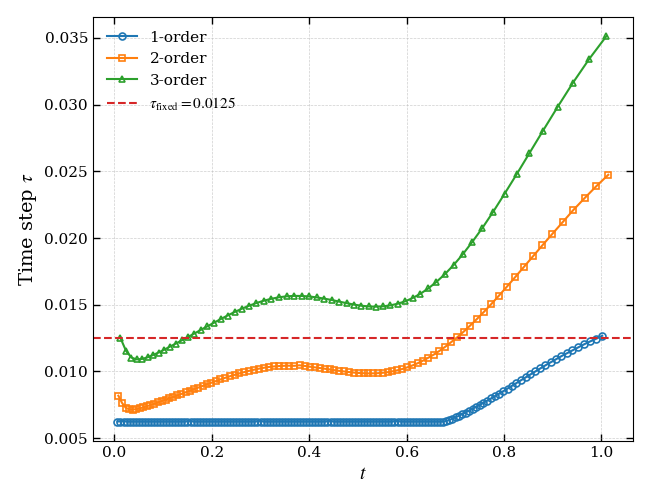}
\caption{First-, second-, and third-order VS-EVNN}
\label{fig:1d-timestep-order}
\end{subfigure}
\caption{One-dimensional Allen--Cahn equation with adaptive time stepping (safety factor $\rho = 0.9$). (a) Adaptive time-step evolution for the first-order scheme at different tolerances. (b) Time-step size for the first-, second-, and third-order schemes at $\mathrm{tol}=0.01$.}
\end{figure}

\begin{table}[H]
\centering
\caption{Accepted-step counts and relative $L^2$ errors of the adaptive VS-EVNN runs for the one-dimensional Allen--Cahn problem (first-, second-, and third-order schemes).}
\label{tab:efficiency_accuracy}
\begin{tabular}{c l c c}
\toprule
Order & Time stepping & Accepted steps & Relative $L^2$ error \\
\midrule
1 & Adaptive, $\mathrm{tol}=0.05$ & 38 & $2.60 \times 10^{-2}$ \\
1 & Adaptive, $\mathrm{tol}=0.02$ & 70 & $1.38 \times 10^{-2}$ \\
1 & Adaptive, $\mathrm{tol}=0.01$ & 144 & $9.24 \times 10^{-3}$ \\
2 & Adaptive, $\mathrm{tol}=0.01$ & 91 & $6.82 \times 10^{-3}$ \\
3 & Adaptive, $\mathrm{tol}=0.01$ & 62 & $2.96 \times 10^{-3}$ \\
\bottomrule
\end{tabular}
\end{table}

\subsubsection{Long horizons for the two-dimensional Allen--Cahn equation} For a long-horizon test, we run the star-shaped problem to $T_{\text{final}}=200$ with the adaptive time-stepping strategy of Section~\ref{subsec:time-adaptive}. The initial time step is $\tau=5$, with a lower bound of $1$ and an upper bound of $10$. The controller tolerance is $\mathrm{tol}=0.08$ and the safety factor is $\rho=0.9$. Figure~\ref{fig:adapt-star-t200-step} shows the energy curves and the time-step variation.

The third-order VS-EVNN energy curve is closest to the reference, while the first- and second-order schemes both show visible errors. All three take small time steps during the initial phase. During a later phase of strong solution variation, the selected time steps shrink.

\begin{figure}[H]
\centering
\begin{subfigure}[t]{0.475\textwidth}
\centering
\includegraphics[width=\linewidth]{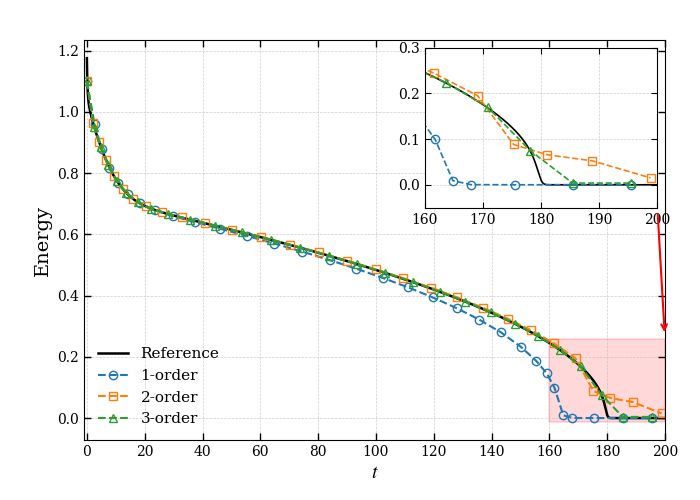}
\caption{Energy evolution}
\end{subfigure}\hfill
\begin{subfigure}[t]{0.475\textwidth}
\centering
\includegraphics[width=\linewidth]{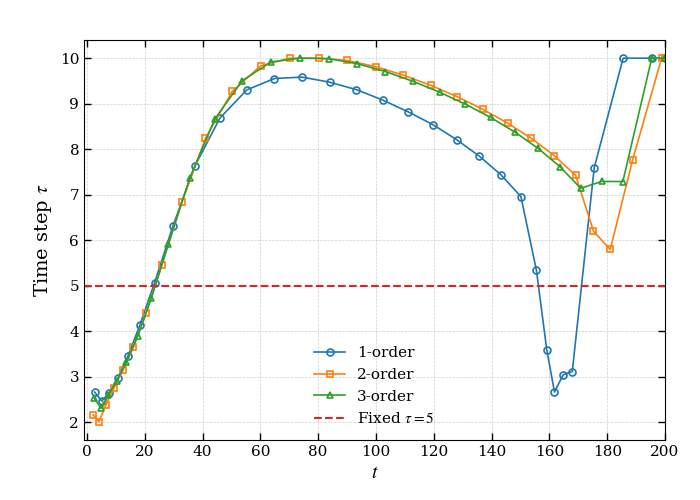}
\caption{Time-step history}
\end{subfigure}
\caption{Two-dimensional Allen--Cahn equation, star-shaped interface, with adaptive time stepping ($T_{\text{final}}=200$). (a) Energy evolution for the first-, second-, and third-order VS-EVNN schemes. (b) Time-step history.}
\label{fig:adapt-star-t200-step}
\end{figure}

\subsubsection{Long horizons for the two-dimensional Cahn--Hilliard equation} We test the adaptive controller of Section~\ref{subsec:time-adaptive} by rerunning the randomly placed circles problem to $T_{\text{final}}=2000$ with adaptive steps for all three orders, alongside fixed-step runs at $\tau_{\text{fixed}}=20$. The adaptive runs use the tolerance $\mathrm{tol}=0.1$, the safety factor $\rho=0.9$, and the step bounds $\tau_{\min}=10$ and $\tau_{\max}=200$. Figure~\ref{fig:2d-ch-adaptive} collects the energy histories, the step-size histories, and the relative $L^2$ error. The controller keeps the step near $\tau_{\text{fixed}}$ during the fast-merging phase, then increases it steadily until the cap $\tau_{\max}$ is reached around $t=1200$. The higher-order schemes grow the step earlier. The energy decays monotonically along the adaptive trajectories, the relative $L^2$ errors of the second- and third-order runs remain on the order of $10^{-3}$ over the whole horizon, and the mass error again stays at the round-off level, as in the one-dimensional diagnostic in Figure~\ref{fig:1d-ch-energy}.

\begin{figure}[H]
\centering
\begin{subfigure}[t]{0.32\textwidth}
\includegraphics[width=\textwidth]{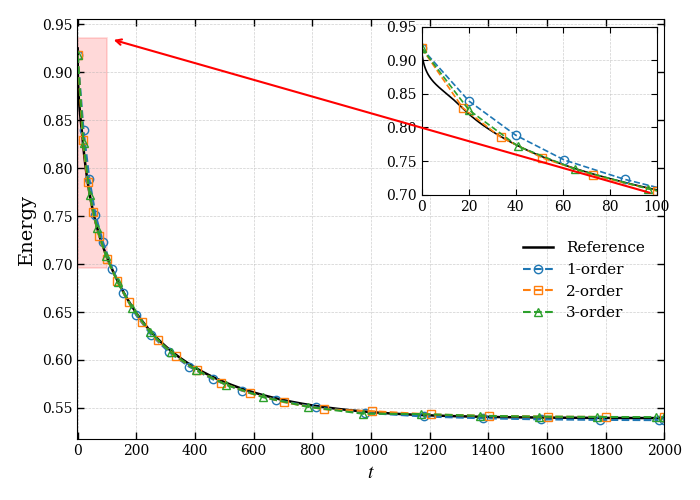}
\caption{Energy evolution}
\end{subfigure}\hfill
\begin{subfigure}[t]{0.32\textwidth}
\includegraphics[width=\textwidth]{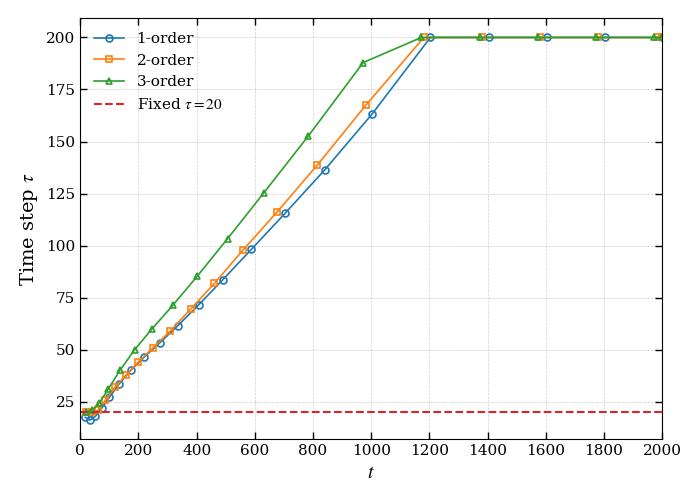}
\caption{Time-step history}
\end{subfigure}\hfill
\begin{subfigure}[t]{0.32\textwidth}
\includegraphics[width=\textwidth]{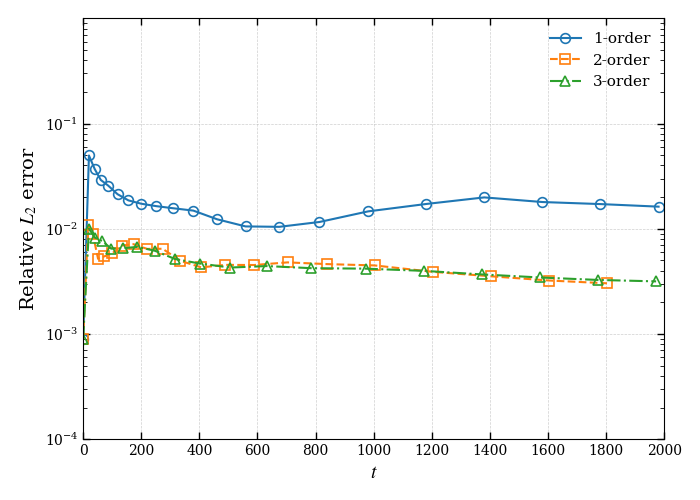}
\caption{Relative $L^2$ error}
\end{subfigure}
\caption{Two-dimensional Cahn--Hilliard equation, randomly placed circles with adaptive time stepping ($T_{\text{final}}=2000$, $\mathrm{tol}=0.1$, $\rho=0.9$, $\tau\in[10,200]$). Energy evolution, time-step history, and relative $L^2$ error for the first-, second-, and third-order adaptive VS-EVNN schemes.}
\label{fig:2d-ch-adaptive}
\end{figure}

\section{Conclusion} 

We have developed a variable-scaled energetic variational neural network (VS-EVNN) method for phase-field gradient flows. The method combines a fixed spatial-coordinate reparametrization with multi-stage variational extrapolation, using warm-started network solves and fixed earlier-stage anchors in the metric of the Allen--Cahn flow or the mobility-weighted, mass-conserving metric of the Cahn--Hilliard flow. Under the reported optimization settings, coordinate scaling and higher-order time integration improve accuracy in the tested examples. The heuristic adaptive controller adjusts the time step based on the observed change in the solution. The numerical experiments exhibit energy decay, and the Cahn--Hilliard mean projection preserves discrete mass up to round-off. For exact stage solves, the variational schemes are unconditionally energy stable.

\section*{Acknowledgments} Xiaobo Jing's work is supported by the National Natural Science Foundation of China (Nos.~12147165 and W2612008), the Jiangsu Provincial Scientific Research Center of Applied Mathematics (No.~BK20233002), the Basic Research Program of Jiangsu (No.~BK20252120), and the Start-up Research Fund of Southeast University (No.~RF1028623369). Jia Zhao acknowledges support from the National Science Foundation under grant NSF-DMS-2513764.

\bibliographystyle{plain}
\bibliography{reference}

\end{document}